\RequirePackage{fix-cm}
\documentclass[smallcondensed]{svjour3}     
\smartqed  
\usepackage[usenames,dvipsnames,svgnames,table]{xcolor}
\usepackage{gastex}

\usepackage{longtable,subfigure}
\usepackage{url,amsmath,amssymb,multirow,enumitem}
\usepackage{array}

\usepackage{pgf}
\usepackage{ifthen}
\usepackage{mathtools}
\usepackage{algorithm}
\usepackage{algpseudocode}

\def\lag{{\textsc{lag}}}
\def\bm{}

\def\e{\epsilon}

\def\Var{\bm{Var}}

\def\argmin{\operatornamewithlimits{argmin}}

 \def\E{{\mathbb E}}
 \def\P{{\mathbb P}}
 \def\chi{{\mathbf 1}} 
  \def\w{\omega } \def\R{{\mathbb R}}
  \def\w{\omega } \def\Z{{\mathbb Z}}
\def\N{{\mathbb N}}  

\def\R{{\mathbb R}}
\def\N{{\mathbb N}}

\DeclarePairedDelimiter\abs{\lvert}{\rvert}%
\makeatletter
\let\oldabs\abs
\def\abs{\@ifstar{\oldabs}{\oldabs*}}

\newtheorem{appl}{Application}
\newtheorem{ex}{Example}

\graphicspath{{./pix/}}

\begin{document}
\title{Algorithm Portfolios for Noisy Optimization}
\author{Marie-Liesse Cauwet \and Jialin Liu \and Baptiste Rozi\`ere \and Olivier Teytaud}

\institute{TAO, INRIA-CNRS-LRI, Univ. Paris-Sud,\\91190 Gif-sur-Yvette, France\\
\email{firstname.lastname@lri.fr\\
https://tao.lri.fr\\
Tel: +33169157643\\
Fax: +33169156579
}
}
\maketitle

{
\begin{abstract}

Noisy optimization is the optimization of objective functions corrupted by noise. A portfolio of solvers is a set of solvers equipped with an algorithm selection tool for distributing the computational power among them. Portfolios are widely and successfully used in combinatorial optimization.

In this work, we study portfolios of noisy optimization solvers. We obtain mathematically proved performance (in the sense that the portfolio performs nearly as well as the best of its solvers) by an ad hoc portfolio algorithm dedicated to noisy optimization. A somehow surprising result is that it is better to compare solvers with some {\em{lag}}, i.e., 
{\color{black}propose the current recommendation of best solver based on their performance }
{\em{earlier}} in the run. An additional finding is a principled method for distributing the computational power among solvers in the portfolio.

\end{abstract}}

\section{Introduction}\label{intro}

Given an objective function, also termed fitness function, from a domain $\mathcal{D}\in \R^{d}$ to $\R$, numerical optimization or simply optimization is the research of points, also termed individuals or search points, with approximately optimum (e.g. minimum) objective function values.

Noisy optimization is the optimization of objective functions corrupted by noise. Black-box noisy optimization is the noisy counterpart of black-box optimization, i.e., functions for which no knowledge about the internal processes involved in the objective function can be exploited.

Algorithm Selection (AS) consists in choosing, in a portfolio of solvers, the one which is approximately the most efficient on the problem at hand. AS can mitigate the difficulties for choosing a priori the best solver among a portfolio of solvers. This means that AS leads to an adaptive version of the algorithms. In some cases, AS outperforms all individual solvers by combining the good properties of each of them, with information sharing or with chaining, as discussed later. It can also be used for the sake of parallelization or parameter tuning, or for mitigating the impact of bad luck in randomized solvers.
In this paper, we apply AS to the black-box noisy optimization problem.
{\color{black}This paper extends \cite{portfolio2} with respect to (i) showing that the {\em{lag}} is necessary; (ii) extending experimental results; (iii) improving the convergence rates thanks to an unfair distribution of the computational budget.
}

\subsection{Noisy optimization}\label{no}

Noisy optimization is a key component of machine learning, from supervised learning to unsupervised 
or reinforcement learning; it is also relevant in streaming applications. The black-box setting is crucial 
in reinforcement learning where gradients are difficult and expensive to get; direct policy search \cite{sutton98} 
usually boils down to (i) choosing a representation and (ii) applying black-box noisy optimization. 
{\color{black}Some works focus on noise models in which the noise has variance close to zero around the 
optimum \cite{augerupu}; others consider actuator noise, i.e., when the fitness values are those at search points corrupted by noise \cite{actuator}, as shown by Eq. \ref{actuat}.
{\color{black}There are also cases in which the landscape is assumed to be rugged, and the optimization algorithm should act as a low pass filter, in order to get rid of local variations \cite{staticnoise}.} The most usual model of noise is however the case of noise in which the variance does not vanish 
around the optimum and the goal is to find a good search point in terms of expected fitness value \cite{fabian,chen1988} 
and this is the case we will consider here. } 

Zero-order methods, including evolution strategies \cite{beyer:book2001} and derivative-free optimization \cite{conn97recent} are natural solutions in the black-box setting; as they do not use gradients, they are not affected in such a setting. However, the noise has an impact even on such methods \cite{stocopti5,noise3}, possibly mitigated by increasing the population size or averaging multiple evaluations of each search point. Using surrogate models \cite{noise2} reduces the impact of noise by sharing information over the domain. Surrogate models are also a step towards higher order methods; even in black-box scenarios: a Hessian can be approximated thanks to observed fitness values and statistical learning of a surrogate model.  

{\color{black}It is known \cite{fabian,shamir} }that stochastic gradient by finite differences (finite differences at each iteration or by averaging over multiple iterations) can provide tight convergence rates (see tightness in \cite{chen1988}) in the case of an additive noise with constant variance. {\color{black}Fabian, in \cite{fabian2},} has also tested the use of second order information (Hessian) approximated by finite differences. 

In this paper, our portfolio will be made of the following algorithms: 
(i) an evolution strategy;
(ii) a first-order method using gradients estimated by finite differences; 
(iii) a second-order method using a Hessian matrix also approximated by finite differences. 
We present these methods in more details in Appendix \ref{newton}. 

\subsection{Algorithm selection}
Combinatorial optimization is probably the most classical application domain for AS \cite{kotthoff}. However, machine learning is also a classical test case \cite{utgoff1988}; in this case, AS is sometimes referred to as meta-learning \cite{aha1992}. 

\paragraph{No free lunch.} \cite{Wolpert1997} claims that it is not possible to do better, on average (uniform average) on all optimization problems from a given finite domain to a given finite codomain. This implies that no AS can outperform existing algorithms on average on this uniform probability distribution of problems. Nonetheless, reality is very different from a uniform average of optimization problems, and AS does improve performance in many cases.

\paragraph{Chaining and information sharing.} Algorithm chaining \cite{borrettTsang1996} means switching from one solver to another during the AS run. More generally, a {\textit{hybrid}} algorithm is a combination of existing algorithms by any means \cite{vassilevska2006}. This is an extreme case of sharing. Sharing consists in sending information from some solvers to other solvers; they communicate in order to improve the overall performance.

\paragraph{Static portfolios \& parameter tuning.} A portfolio of solvers is usually static, i.e., combines a finite number of given solvers. {\color{black}SATzilla is probably the most well known portfolio method, combining several SAT-solvers \cite{xu2008satzilla}. Samulowitz and Memisevic have pointed out in \cite{samulowitz2007}  the importance of having ``orthogonal'' solvers in the portfolio, so that the set of solvers is not too large, but approximates as far as possible the set of all feasible solvers. AS and parameter tuning are combined in \cite{xuhydra2010}; parameter tuning can be viewed as an AS over a large but structured space of solvers. We refer to \cite{kotthoff} and references therein for more information on parameter tuning and its relation to AS; this is beyond the scope of this paper.}

\paragraph{Fair or unfair distribution of computation budgets.} In \cite{pulinaTacchella2009}, different strategies are compared for distributing the computation time over different solvers. 
The first approach consists in running all solvers during a finite time, then selecting the best performing one, and then keeping it for all the remaining time. Another approach consists in running all solvers with the same time budget independently of their performance on the problem at hand. Surprisingly enough, they conclude that uniformly distributing the budget is a good and robust strategy. The situation changes when a training set is available, and when we assume that the training set is relevant for the future problems to be optimized; \cite{Kadioglu2011}, using a training set of problems for comparing solvers, proposes to use 90\% of the time allocated to the best performing solver, the other 10\% being equally distributed among other solvers. {\color{black}In \cite{gaglioloSchmidhuber2005,gaglioloSchmidhuber2006b}, it is proposed to use} 50\% of the time budget for the best solver, 25\% for the second best, and so on. Some AS algorithms \cite{gaglioloSchmidhuber2006b,armstrong2006} do not need a separate training phase, and perform entirely online solver selection; a weakness of this approach is that it is only possible when a large enough budget is available, so that the training phase has a minor cost. {\color{black}A portfolio algorithm, namely Noisy Optimization Portfolio Algorithm (NOPA), designed for noisy optimization solvers, and which distributes uniformly the computational power among them, is proposed in \cite{portfolio2}. We extend it to INOPA (Improved NOPA), which is allowed to distribute the budget in an unfair manner.} It is proved that INOPA reaches the same {\color{black}convergence rate} as the best solver, within a small (converging to $1$) multiplicative factor on the number of evaluations, when there is a unique optimal solver - thanks to a principled distribution of the budget into (i) running all the solvers; (ii) comparing their results; (iii) running the best performing one. The approach is anytime, in the sense that the computational budget does not have to be known in advance.

\paragraph{Parallelism.} 
We refer to \cite{hamadi2013combinatorial} for more on parallel portfolio algorithms (though not in the noisy optimization case).
Portfolios can naturally benefit from parallelism; however, the situation is different in the noisy case, which is highly parallel by nature (as noise is reduced by averaging multiple resamplings\footnote{{``Resamplings'' means that the stochastic objective function, also known as fitness function, is evaluated several times at the same search point. This mitigates the effects of noise.}}). 

\paragraph{Best solver first.} \cite{pulinaTacchella2009} point out the need for a good ordering of solvers, even if it has been decided to distribute nearly uniformly the time budget among them{\color{black}: this improves the anytime behavior}. For example, they propose, within a given scheduling with same time budget for each optimizer, to use first the best performing solver. We will adapt this idea to our context; this leads to INOPA, improved version of NOPA.

\paragraph{Bandit literature.} During the last decade, a wide literature on bandits \cite{lairobbins,Auer03,BubeckE09} has proposed many tools for distributing the computational power over stochastic options to be tested. The application to our context is however far from being straightforward. In spite of some adaptations to other contexts (time varying as in \cite{kocsis} or adversarial \cite{grigoriadis,auer95gambling}), and maybe due to strong differences such as the very non-stationary nature of bandit problems involved in optimization portfolios, these methods did not, for the moment, really find their way to AS. 
Another approach consists in writing this bandit algorithm as a meta-optimization problem; \cite{stpierre:hal-00979456} applies the differential evolution algorithm \cite{referenceOnDE} to some non-stationary bandit problem, which outperforms the classical bandit algorithm on an AS task.

{The main contributions of this paper can be summarized as follows. First, we prove positive results for a portfolio algorithm, termed NOPA, for AS in noisy optimization. Second, we design a new AS, namely INOPA, which (i) gives the priority to the best solvers when distributing the computational power; (ii) approximately reaches the same performance as the best solver; 
(iii) possibly shares information between the different solvers. We then prove the requirement of selecting the solver that was apparently the best {\textit{some time before}} the current iteration - a phenomenon that we term the {\em{lag}}. Finally, we provide some experimental results.}

\subsection{Outline of this paper}
Section \ref{portfolio} describes the algorithms under consideration and provides theoretical results.
Section \ref{xp} is dedicated to experimental results.
Section \ref{sec:conc} concludes.

\section{Algorithms and analysis}\label{portfolio}
Section \ref{notation} introduces some notations.
Section \ref{definition} provides some background and criteria.
Section \ref{descriptionofportfolioalgorithm} describes two portfolio algorithms, one with fair distribution of budget among solvers and one with unfair distribution of budget.
Section \ref{theorie} then provides theoretical guarantees.

\subsection{Notations}\label{notation}

{\color{black}In this paper, $\N^*=\{1,2,3,\dots\}$, ``a.s.'' stands for ``almost surely'', i.e., with probability 1, and ``s.t.'' stands for ``such that''. Appendix \ref{ion} is a summary of notations.}
If $X$ is a random variable, then $(X^{(1)},X^{(2)},\dots)$ denotes a sample of independent identically distributed random variables, copies of $X$. {\color{black}$o(.),O(.),\Theta(.)$ are the standard Landau notations. {$\cal N$ denotes a standard Gaussian distribution, in dimension given by the context.}

Let $f : \cal{D}\rightarrow \R$ be a noisy function. 
$f$ is a random process, and equivalently it can be viewed as a mapping $(x,\w)\mapsto f(x,\w)$ where $x\in \cal{D}$ and $\w$ is a random variable independently sampled at each call to $f$. The user can only choose $x$. For short, we will use the notation $f(x)$. The reader should keep in mind that this function is stochastic.
{$\hat{\E}_{s}\left[f(x)\right]$ denotes the empirical evaluation of $\E\left[f(x)\right]$ over $s\in\N^*$ resamplings, i.e., $\hat{\E}_{s}\left[f(x)\right]=\frac{1}{s}\sum_{j=1}^{s}\left(f(x)\right)^{(j)}$.}

\subsection{Definitions and Criteria}\label{definition}
A black-box noisy optimization solver, here referred to as a solver, is a program which aims at finding the minimum $x^*$ of $x\mapsto \E f(x)$, thanks to multiple black-box calls to the unknown function $f$. The expectation operator $\E$ shows that we assume that the noise is additive and unbiased (Eq. \ref{additivenoise}); the ``real'', noise-free, fitness function is the expectation of the noisy fitness function. This is not necessarily the case for e.g. actuator noise, as in Eq. \ref{actuat}:
\begin{eqnarray}
f(x,\w)=f(x)+\w;\label{additivenoise}\\
f(x,\w)=f(x+\w).\label{actuat} 
\end{eqnarray}

The portfolio algorithm has the same goal, and can use $M\in\{2,3,\dots\}$ different given solvers. A good AS tool should ensure that it is nearly as efficient as the best of the individual solvers\footnote{{A solver is termed ``individual solver'' when it is not a portfolio. In this paper, unless stated otherwise, a ``solver'' is an ``individual solver''.}}, for any problem in some class of interest.

\paragraph{Simple regret criterion.}
In the black-box setting, let us define :
\begin{itemize}
	\item $x_n$ the $n^{th}$ search point at which the objective function (also termed fitness function) is evaluated;
	\item $\tilde x_n$ the point that the solver recommends as an approximation of the optimum after having evaluated the objective function at $x_1,\dots,x_n$ (i.e., after spending $n$ evaluations from the budget).
\end{itemize}
Some algorithms make no difference between $x_n$ and $\tilde x_n$, but in the general case of a noisy optimization setting the difference matters \cite{CLOP,fabian,shamir}.

The Simple Regret (SR) for noisy optimization is expressed in terms of objective function values, as follows:
\begin{equation}
SR_n=\E \left( f(\tilde x_n)-f(x^*) \right).\label{eq:eqsr} 
\end{equation}
$SR_n$ is the simple regret after $n$ evaluations; $n$ is then the budget.
The $\E$ operator refers to the $\w$ part, i.e., with complete notations,
$$SR_n=\E_\w \left( f(\tilde x_n,\w)-f(x^*,\w) \right). $$

{
In many cases, it is known that the simple regret has a linear convergence in a $\log$-$\log$ scale \cite{fabian,chen1988,CLOP}. Therefore we will consider this slope.
The slope of the simple regret is then defined as
\begin{equation}
s(SR) = \underset{n\rightarrow \infty}{\lim} \frac{\log(SR_n)}{\log (n)},\label{eqssr}
\end{equation}}
{\color{black} where the limit holds almost surely, since $SR_n$ is a random variable.} 

For example, the gradient method proposed in \cite{fabian} (approximating the gradient by finite differences) 
reaches a simple regret slope arbitrarily close to $-1$ on sufficiently smooth problems, for an additive centered noise, without assuming variance decreasing to zero around the optimum.

\paragraph{Simple regret criterion for portfolio.}
For a portfolio algorithm in the black-box setting, $\forall i \in \{1,\dots,M\}$, $\tilde{x}_{i,n}$ denotes the point 
\begin{itemize}
	\item that the solver number $i$ {\em{recommends}} as an approximation of the optimum;
	\item after this solver has spent $n$ evaluations from the budget.
\end{itemize}

Similarly, the simple regret given by Equation \ref{eq:eqsr} corresponding to solver number $i$ after $n$ evaluations (i.e., after solver number $i$ has spent $n$ evaluations), is denoted by $SR_{i,n}$. 
For $n\in \N^*$, $i^*_n$ denotes the solver chosen by the selection algorithm when there are at most $n$ evaluations per solver.\footnote{This is not uniquely defined, as there might be several time steps at which the maximum number of evaluations in a solver is $n$; however, the results in the rest of this paper are independent of this subtlety.}

Another important concept is the {\color{black}difference between the two kinds of terms in the regret of the portfolio. We distinguish these two kinds of terms in the next two definitions:}

\begin{definition}[Solvers' regret]
The solvers' regret with index $n$, denoted by $SR_{n}^{Solvers}$, is the minimum simple regret among the solvers after at most $n$ evaluations each, i.e., $SR_{n}^{Solvers}=\underset{i\in \{1,\dots M\}}{\min} SR_{i,n}$.
\end{definition}

\begin{definition}[Selection regret]\label{def:selsr}
The selection regret with index $n$, denoted by $SR_{n}^{Selection}$ includes the additional regret due to mistakes in choosing among these $M$ solvers after at most $n$ evaluations each, i.e., $SR_{n}^{Selection}=\E \left( f(\tilde x_{i^{*}_{n},n})-f(x^*) \right)$. 
\end{definition}

Similarly, $\Delta_{i,n}$ quantifies the regret for choosing solver $i$ at iteration $n$.
\begin{definition}
For any solver $i \in \{1,\dots,M\}$ and any number of evaluations $n\in \N^{*}$, we denote by $\Delta_{i,n}$ the quantity: $\Delta_{i,n} = SR_{i,n}- \underset{j \in \{1,\dots,M\}}\min SR_{j,n}$.
\end{definition}

Finally, we consider a function that will be helpful for defining our portfolio algorithms.

\begin{definition}[$\lag$ function]\label{def:lag}
A lag function $\lag : \N^*\rightarrow\N^*$ is a non-decreasing function
such that for all $n\in \N^*$, $\lag(n)\leq n$.
\end{definition}

\subsection{Portfolio algorithms}\label{descriptionofportfolioalgorithm}
In this section, we present two AS methods. A first version, in Section \ref{nopaunif}, shares the computational budget uniformly; a second version has an unfair sharing of computation budget, in Section \ref{inopa}.

\subsubsection{Simple Case : Uniform Portfolio NOPA}\label{nopaunif}
We present in Algorithm \ref{algo:nopa} a simple noisy optimization portfolio algorithm (NOPA) which does not apply any sharing and distributes the computational budget equally over the noisy optimization solvers. 
\begin{algorithm}
	{\color{black}
\begin{algorithmic}[1]
\Require {noisy optimization solvers $Solver_1, Solver_2,\dots, Solver_M$}
\Require {a {\em{lag}} function $\lag: \N^{*} \mapsto \N^{*}$} \Comment{As in Definition \ref{def:lag}}
\Require {a non-decreasing integer sequence $r_1, r_2, \dots$ \Comment{Periodic comparisons}}
{\color{black}\Require {a non-decreasing integer sequence $s_1, s_2, \dots$} \Comment{Number of resamplings}}
\State{$n\leftarrow 1$} \Comment{Number of selections}
\State{$m\leftarrow 1$}	 \Comment{NOPA's iteration number}
\State{$i^*\leftarrow null$}	 \Comment{Index of recommended solver}
\State{$x^*\leftarrow null$}	 \Comment{Recommendation}
\While{budget is not exhausted}
	\If{$m \geq r_n$}
		\State{$i^{*}=\underset{i \in \{1,\dots,M\}}{\arg\min} \hat \E_{s_n}[f(\tilde{x}_{i,\lag(r_{n})})]$}\label{line:comparison} \Comment{Algorithm selection}
    \State{$n\leftarrow n+1$}
	\Else
		\For{$i \in \{1,\dots,M\}$}\label{line:beginrun}
			\State{Apply one evaluation for $Solver_i$}
		\EndFor\label{line:endrun}
		\State{$m\leftarrow m+1$}
	\EndIf
	\State{$x^*=\tilde x_{i^{*},m}$} \Comment{Update recommendation}
\EndWhile
\end{algorithmic}}
\caption{\label{algo:nopa} Noisy Optimization Portfolio Algorithm (NOPA).}
\end{algorithm} 

\def\nouse{
\begin{figure*}
\centering
\subfigure[Evaluate each solver which received less than $r_n$ evaluations until they all have at least $r_n$ evaluations. Make a selection based on the evaluations at iteration $\lag(r_n)$, e.g., assuming $Solver_2$ is selected (yellow circle), recommend $\tilde x_{2,r_n}$ (red square).]{\includegraphics[width=.45\textwidth]{nopa1-pics-crop.pdf}}
\hspace{4mm}
\subfigure[Evaluate each solver which received less than $r_{n+1}$ evaluations. Make a selection based on the evaluations at iteration $\lag(r_{n+1})$, e.g., assuming $Solver_3$ is selected (yellow circle), recommend $\tilde x_{3,r_{n+1}}$ (red square).]{\includegraphics[width=.45\textwidth]{nopa2-pics-crop.pdf}} 
\caption{\label{fig:nopa}An example of NOPA with 4 solvers until reaching iteration $r_n$ (a) and iteration $r_{n+1}$ (b).}
\end{figure*}
}

{In this NOPA algorithm, we compare, at iteration $r_n$, recommendations chosen at iteration $\lag(r_n)$, and this comparison is based on $s_n$ resamplings, where $n$ is the number of algorithm selection steps. 
	{\color{black}We have designed the algorithm as follows:}
\begin{itemize}
{\color{black}
	\item {\bf{A stable choice of solver:}} The selection algorithm follows the recommendation of the same solver $i^*$ at all iterations in $\{r_n,\dots,r_{n+1}-1\}$. This choice is based on comparisons between old recommendations (through the {\em{lag}} function $\lag$).
	\item {\bf{The chosen solver updates are taken into account.}} For iteration indices $m <p$ in $\{r_n,\dots,r_{n+1}-1\}$, the portfolio chooses the same solver $i^*$, but does not necessarily recommends the same point because possibly the solver changes its recommendation, i.e., possibly $\tilde x_{i^*,m}\neq \tilde x_{i^*,p}$.}
\end{itemize}}

\paragraph{Effect of the lag.}
{\color{black}{Due to the $\lag(.)$ function, we compare the $\tilde x_{i,\lag(r_n)}$ (for $i\in \{1,\dots,M\}$), and not the
		$\tilde x_{i,r_n}$. This is the key point of this algorithm. Comparing the $\tilde x_{i,\lag(r_n)}$ is much cheaper than comparing the $\tilde x_{i,r_n}$, because the fitness values are not yet that good at iteration $\lag(r_n)$, so they can be compared faster - i.e., with less evaluations - than recommendations at iteration $r_n$. We will make this more formal in Section \ref{theorie}, and see under which assumptions this lag has more pros than cons, namely when algorithms have somehow sustained rates. {\color{black}In addition, with lag, we can define INOPA, which saves up significant parts of the computation time.}

The first step for formalizing this is to understand the two different kinds of evaluations in portfolio algorithms for noisy optimization. Contrarily to noise-free settings, comparing recommendations requires a dedicated budget, which is far from negligible.}}
It follows that there are two kinds of evaluations:
\begin{itemize}
	\item {\bf{Portfolio budget ({\color{black}Algorithm \ref{algo:nopa}, Lines \ref{line:beginrun}-\ref{line:endrun}}):}} this corresponds to the $M$ evaluations per iteration, dedicated to running the $M$ solvers (one evaluation per solver and per iteration).
	\item {\bf{Comparison budget ({\color{black}Algorithm \ref{algo:nopa}, Line \ref{line:comparison}}):}} this corresponds to the $s_n$ evaluations per solver at the $n^{th}$ algorithm selection. This is a key difference with deterministic optimization. In deterministic optimization, this budget is zero as the exact fitness value is readily available.
\end{itemize}

We have $M\cdot r_n$ evaluations in the portfolio budget for the first $r_n$ iterations. We will see below (Section \ref{theorie}) conditions under which the other costs (i.e. comparison costs) can be made negligible, whilst preserving the same regret as the best of the $M$ solvers.

\subsubsection{INOPA: Improved NOPA, with unequal budget}\label{inopa}
Algorithm \ref{algo:inopa} proposes a variant of NOPA, which distributes the budget in an unfair manner. 
The solvers with good performance receive a greater budget. The algorithm is designed so that it mimics the behavior of NOPA, but without spending the evaluations which are useless for the moment, given the lag - i.e. we use the fact that evaluations prior to the lagged index are useless except for the selected algorithm.

\begin{algorithm}
	{\color{black}
\begin{algorithmic}[1]
\Require {noisy optimization solvers $Solver_1, Solver_2,\dots, Solver_M$}
\Require {a {\em{lag}} function $\lag: \N^{*} \mapsto \N^{*}$} \Comment{Refer to Definition \ref{def:lag}}
\Require {a non-decreasing positive integer sequence $r_1, r_2, \dots$}{\color{black}\Comment{Periodic comparisons}}
{\color{black}\Require {a non-decreasing integer sequence $s_1, s_2, \dots$} \Comment{Number of resamplings}}
\State{$n\leftarrow 1$} \Comment{Number of selections}
\State{$m\leftarrow 1$}	 \Comment{NOPA's iteration number}
\State{$i^*\leftarrow null$}	 \Comment{Index of recommended solver}
\State{$x^*\leftarrow null$}	 \Comment{Recommendation}
\While{budget is not exhausted}
	\If{$m \geq \lag(r_n)$ or $i^*=null$}
		\State{$i^{*}=\underset{i \in \{1,\dots,M\}}{\arg\min} \hat \E_{s_n}[f(\tilde{x}_{i,\lag(r_{n})})]$} \Comment{Algorithm selection}
		\State{$m'\leftarrow r_n$}
		\While{$m' < r_{n+1}$}
			\State{Apply one evaluation to solver $i^*$}
			\State{$m'\leftarrow m'+1$}
			\State{$x^*=\tilde x_{i^{*},m'}$} \Comment{Update recommendation}
		\EndWhile
		\State{$n\leftarrow n+1$}
	\Else
		\For{$i \in \{1,\dots,M\}\backslash i^*$}
			\State{Apply $\lag(r_n) - \lag(r_{n-1})$ evaluations for $Solver_i$}
		\EndFor
		\State{$m\leftarrow m+1$}
	\EndIf
\EndWhile
\end{algorithmic}
}
\caption{\label{algo:inopa} Improved Noisy Optimization Portfolio Algorithm (INOPA).}
\end{algorithm} 
\def\nouse{
\begin{figure*}
\centering
\subfigure[Evaluate each solver which received less than $\lag(r_1)$ evaluations, until they all have at least $\lag(r_1)$ evaluations. Make a selection by comparing recommendations obtained by solvers at the end of their first $\lag(r_1)$ evaluations (yellow circle).]{\includegraphics[width=.45\textwidth]{inopa1-pics-crop.pdf}}
\hspace{4mm}
\subfigure[Evaluate the best solver, e.g. possibly $Solver_2$, until it has received $r_1$ evaluations, and recommend $\tilde x_{2,r_1}$ (red square).]{\includegraphics[width=.45\textwidth]{inopa2-pics-crop.pdf}} \\
\subfigure[Evaluate each solver which received less than $\lag(r_2)$ evaluations. Make a selection at iteration $\lag(r_2)$ (yellow circle).]{\includegraphics[width=.45\textwidth]{inopa3-pics-crop.pdf}}
\hspace{4mm}
\subfigure[Evaluate the best solver, e.g. supposing $Solver_3$, until receiving $r_2$ evaluations, recommend $\tilde x_{3,r_2}$ (red square).]{\includegraphics[width=.45\textwidth]{inopa4-pics-crop.pdf}}
\caption{\label{fig:inopa}An example of INOPA with 4 solvers during the first and the second selection. For algorithm selection time $>2$, repeat the steps (b)-(d).}
\end{figure*}
}

\def\oldalgo{
\begin{algorithm}
\caption{\label{algo:inopa} Improved Noisy Optimization Portfolio Algorithm (INOPA).}
{\bf{Iteration 1:}} one evaluation for solver $1$, one evaluation for solver $2$, \dots, one evaluation for solver $M$.\\
{\bf{Iteration 2:}} one evaluation for solver $1$, one evaluation for solver $2$, \dots, one evaluation for solver $M$.\\

\makebox[\linewidth][c]{$\smash{\vdots}$\ \ \ $\smash{\vdots}$\ \ \ $\smash{\vdots}$ }\\

{\bf{Iteration $\lag(r_1)$:}} one evaluation for solver $1$, one evaluation for solver $2$, \dots, one evaluation for solver $M$.\\
{\fbox{\bf{Algorithm Selection}}} Evaluate $X=\{\tilde x_{1,\lag(r_1)},\dots,\tilde x_{M,\lag(r_1)}\}$, each of them $s_1$ times; the AS selects {$i^{*}=\underset{i \in \{1,\dots,M\}}{\argmin}\hat \E_{s_1}[f(\tilde{x}_{i,\lag(r_{1})})]$}. Evaluate solver $i^{*}$ until it has $r_1$ evaluations. For $m \in \{r_1,\dots,r_2-1\}$, the recommendation of the selection algorithm is $\tilde x_{i^*,m}$ i.e., $i^*_m=i^*$.\\
{\bf{Iteration $\lag(r_1)+1$:}} one evaluation for each solver which received less than $\lag(r_2)$ evaluations.\\

\makebox[\linewidth][c]{$\smash{\vdots}$\ \ \ $\smash{\vdots}$\ \ \ $\smash{\vdots}$ }\\

{\bf{Iteration $\lag(r_2)-1$:}} one evaluation for each solver which received less than $\lag(r_2)$ evaluations.\\
{\bf{Iteration $\lag(r_2)$:}} one evaluation for each solver which received less than $\lag(r_2)$ evaluations.\\
 {\fbox{\bf{Algorithm Selection}}} Evaluate $X=\{\tilde x_{1,\lag(r_2)},\dots,\tilde x_{M,\lag(r_2)}\}$, each of them $s_2$ times; the AS selects {$i^*=\underset{i \in \{1,\dots,M\}}{\arg\min} \hat \E_{s_2}[f(\tilde{x}_{i,\lag(r_{2})})]$}. Evaluate solver $i^{*}$ until it has $r_2$ evaluations. For $m \in \{r_2,\dots,r_3 -1\}$, the recommendation of the selection algorithm is $\tilde x_{i^*,m}$, i.e., $i^*_m=i^*$.\\

\makebox[\linewidth][c]{$\smash{\vdots}$\ \ \ $\smash{\vdots}$\ \ \ $\smash{\vdots}$ }\\

{\bf{Iteration $\lag(r_n)$:}} one evaluation for each solver which received less than $\lag(r_n)$ evaluations.\\
{\fbox{\bf{Algorithm Selection}}} Evaluate $X=\{\tilde x_{1,\lag(r_n)},\dots,\tilde x_{M,\lag(r_n)}\}$, each of them $s_n$ times; the AS selects {$i^{*}=\underset{i \in \{1,\dots,M\}}{\arg\min} \hat \E_{s_n}[f(\tilde{x}_{i,\lag(r_{n})})]$}. Evaluate solver $i^{*}$ until it has $r_n$ evaluations. For $m \in \{r_{n},\dots,r_{n+1}-1\}$, the recommendation of the selection algorithm is $\tilde x_{i^*,m}$, i.e., $i^*_m=i^*$.\\
\end{algorithm}
}

\subsection{Theoretical analysis}\label{theorie}
We here show 
\begin{itemize}
	\item a bound on the performance of NOPA (Section \ref{sec:logMshift});
	\item a bound on the performance of INOPA (Section \ref{logmp});
	\item that the {\em{lag}} term is necessary (Section \ref{neclag}).
\end{itemize}

\subsubsection{Preliminary}\label{sec:preliminary}
We define 2 extra properties which are central in the proof.

\begin{definition}[${\bf{P_{as}^{(i)}((\epsilon_n)_{n\in \N^*})}}$]\label{hpsi}
For any solver $i\in \{1,\dots,M\}$, for some positive sequence $(\e_n)_{n\in \N^{*}}$, we define ${\bf{P_{as}^{(i)}((\epsilon_n)_{n\in \N^*})}}$:
\begin{equation*}
{{P_{as}^{(i)}((\epsilon_n)_{n\in \N^*})}}:
a.s.\quad  \exists n_0,~\forall n_1 \ge n_0,~\Delta_{i,n_1}< 2 \epsilon_{n_1} \implies \forall n_2 \ge n_1,~\Delta_{i,n_2}<2\epsilon_{n_2}.
\end{equation*}
\end{definition}

Informally speaking, if ${\bf{P_{as}^{(i)}((\epsilon_n)_{n\in\N^*})}}$ is true, then almost surely for a large enough number of evaluations, the difference between the simple regret of solver $i\in \{1,\dots,M\}$ and the optimal simple regret is either always at most $2\epsilon_n$ or always larger - there is no solver infinitely often alternatively at the top level and very weak.

\begin{definition}[${\bf{P_{as}((\epsilon_n)_{n\in \N^*})}}$]\label{hps}
For some positive sequence $(\e_n)_{n\in \N^{*}}$, we define ${\bf{P_{as}((\epsilon_n)_{n \in \N^{*}})}}$ as follows:
\begin{equation*}
\forall i \in \left\{1,\dots,M \right\},\ {\bf{P_{as}^{(i)}((\epsilon_n)_{n \in \N^{*}})}}\text{ holds}.
\end{equation*}
\end{definition}

\begin{remark}
In Definitions \ref{hpsi} and \ref{hps}, we might choose slightly less restrictive definitions, for which the inequalities only hold for integers $n$ such that $\exists i$, $\lag(r_i)=n\mbox{ or }r_i=n$.
\end{remark}

Definitions above can be applied in a very general setting.
The simple regret of some noisy optimization solvers, for instance Fabian's algorithm, 
is almost surely $SR_{n}\leq (1+ o(1)) \frac {C} {n^{\alpha}}$ after $n\in \N^*$ evaluations ($C$ is a constant), 
for some constant $\alpha>0$ arbitrarily close to $1$. {\color{black}This result is proved in \cite{fabian}, with 
optimality proved in \cite{chen1988}.} For RSAES, introduced below, the proof is given in \cite{loglog}. For noisy variants of Newton's algorithm, one can refer to \cite{spall00adaptive,spall09feedback,tcscauwet}.

We prove the following proposition for such a case; it will be convenient for illustrating ``abstract'' general results to standard noisy optimization frameworks.

\begin{proposition}\label{example} \label{proposition}
Assume that each solver $i \in \{1,\dots,M\}$ has almost surely simple regret $(1+ o(1)) \frac {C_i} {n^{\alpha_i}}$ after $n\in \N^*$ evaluations.

We define $C$, $\alpha^*$, $C^*$:
\begin{equation}
C= \frac 1 3 \min \left\{|C_i - C_j| \mid 1 \le i,j \le M; C_i-C_j \ne 0 \right\}. \label{eq:defc}
\end{equation}
\begin{equation}
	\alpha^*= \underset{i \in \{1,\dots,M\}}\max \alpha_i  \label{eq:bestalpha}.
\end{equation}
\begin{equation}
	C^* = \underset{i \in \{1,\dots,M\} \ s.t.\ \alpha_i = \alpha^* }\min C_i.   \label{eq:bestc}
\end{equation}

We also define the set of optimal solvers:
\begin{eqnarray}
	SetOptim&=&\{i \in \{1,\dots,M\} \vert \alpha_i = \alpha^*\}\nonumber\\
	\mbox{ and }SubSetOptim &=& \{i^* \in SetOptim \vert C_{i^*} = C^* \}\label{ssopt}\\
	&=&\{i \in \{1,\dots,M\} \vert \alpha_i = \alpha^*\text{ and } C_{i} = C^* \}\label{subsetoptim}.
\end{eqnarray}

With these notations, if almost surely, $\forall i \in \{1, \dots, M\}$, the simple regret for solver $i$ after $n\in \N^*$ evaluations is $SR_{i,n}= (1+ o(1)) \frac {C_i} {n^{\alpha_i}}$, then ${\bf{P_{as}((\epsilon_n)_{n \in \N^{*}})}}$ is true with $\e_n$ defined as follows:
\begin{equation}
\epsilon_n= \frac C { n^{\alpha^*}}  \label{eq:epsilon}.
\end{equation}

Moreover, if $ i_0 \in \{1,\dots,M\}$ satisfies: $\left(\exists n_0\in \N^*,\ \forall n \geq n_0,\ \Delta_{i_0,n}\leq 2\e_n\right)$, then $i_0 \in SubSetOptim$.
\end{proposition}

Informally speaking, this means that if the solver $i_0$ is close, in terms of simple regret, to an optimal solver (i.e., a solver matching $\alpha^*$ and $C^*$ in Equations {\ref{eq:bestalpha} and \ref{eq:bestc}), then it also has an optimal slope ($\alpha_{i_0}=\alpha^*$) and an optimal constant ($C_{i_0}=C^{*}$). 

\begin{proof}
For any solver $i \in \{1,\dots,M\}$ and any solver $i^* \in SubSetOptim$,
\begin{equation}
SR_{i,n} -SR_{i^*,n} = (1+o(1))\frac {C_i} {n^{\alpha_i}} - (1+o(1))\frac {C^*} {n^{\alpha^{*}}}.	\label{eq:diffsr}
\end{equation}
By Equations \ref{eq:epsilon} and \ref{eq:diffsr},
\begin{equation}
\frac {SR_{i,n} -SR_{i^*,n}} {\epsilon_n} = \frac { C_i} C \cdot n^{\alpha^* - \alpha_i} \cdot(1+o(1)) - \frac {C^*} C \cdot(1+o(1)). \label{eq:ce}
\end{equation}

\begin{itemize}
\item If $i \not \in SetOptim$, i.e., $\alpha_i < \alpha^{*}$, the first term in Equation \ref{eq:ce} tends to $\infty$, which leads to
\begin{equation}
\underset{\alpha_i < \alpha^{*},n \rightarrow \infty}\lim \frac {SR_{i,n} -SR_{i^*,n}} {\epsilon_n} = \infty.\nonumber
\end{equation}
So for all $i \not \in SetOptim$, $\exists n_0\in \N^*$ s.t. $\forall n \ge n_0$, $\Delta_{i,n} = SR_{i,n} - \underset{j \in \{1,\dots,M\}}\min SR_{j,n} > 2 \epsilon_n$ and, therefore, ${\bf{P_{as}^{(i)}((\epsilon_n)_{n\in \N^*})}}$ is true.

\item If $i \in SetOptim$, i.e., $\alpha_i = \alpha^*$, Equation \ref{eq:ce} becomes
\begin{equation*}
\frac {SR_{i,n} -SR_{i^*,n}} {\epsilon_n} = \frac{C_i - C^*} C + \frac{ C_{i}} C o(1) - \frac{C^*} C o(1)
\end{equation*}
and therefore
\begin{equation}
\label{allonsallonsunpeudhygiene}
\underset{n\rightarrow \infty} \lim\frac {SR_{i,n} -SR_{i^*,n}} {\epsilon_n} = \frac {C_i- C^*} C.\nonumber
\end{equation}

\begin{itemize}
\item If $i \in SubSetOptim$, i.e., $C_i =C^*$, $\underset{n\rightarrow \infty} \lim \frac {SR_{i,n} -SR_{i^*,n}} {\epsilon_n} = 0$. Therefore, ${\bf{P_{as}^{(i)}((\epsilon_n)_{n\in \N^*})}}$ is true.
\item if $i \notin SubSetOptim$, $\underset{n\rightarrow \infty} \lim\frac {SR_{i,n} -SR_{i^*,n}} {\epsilon_n} \ge 3$ by definition of $C$ (Equation \ref{eq:defc}). Therefore, ${\bf{P_{as}^{(i)}((\epsilon_n)_{n\in \N^*})}}$ is true. 
\end{itemize}
So for all $i \in \{1,\dots,M\}$, ${\bf{P_{as}^{(i)}((\epsilon_n)_{n\in \N^*})}}$ is true, hence $\bf{P_{as}((\epsilon_n)_{n\in \N^*})}$ holds.

Moreover, it shows that $\exists n_0 \in \N^*$, $\forall n \ge n_0$, $SR_{n}^{Solvers} = \underset{j \in \{1,\dots,M\}}\min SR_{j,n}= SR_{j^*,n}$ where $j^* \in SubSetOptim$.
\end{itemize}
\qed\end{proof}

\subsubsection{The $\log(M)$-shift for NOPA}\label{sec:logMshift}

We can now enunciate the first main theorem, stating that there is, with fair sharing of the budget as in NOPA, a $\log(M)$-shift, i.e., on a $\log$-$\log$ scale (x-axis equal to the number of evaluations and y-axis equal to the log of the simple regret), the regret of the portfolio is just shifted by $\log(M)$ on the x-axis.

\begin{theorem}[Regret of NOPA: the $\log(M)$ shift]\label{rnopa}
Let $(r_n)_{n\in \N^{*}}$ and $(s_n)_{n\in \N^{*}}$ be two non-decreasing integer sequences. 
Assume that:   
\begin{itemize}
\item $\forall x \in \mathcal{D},\ Var\ f(x)\leq 1$;
\item for some positive sequence $(\e_n)_{n\in \N^{*}}$, ${\bf{P_{as}((\epsilon_n)_{n \in \N^*})}}$ (Definition \ref{hps}) is true. 

Then, there exists $n_0$ such that:
\begin{equation}\label{eq:ecartSR}
\forall n\geq n_0,\ SR_{r_n}^{Selection} < SR_{r_n}^{Solver} + 2\e_{r_n}
\end{equation}
\begin{equation}\label{eq:proba}
  \mbox{with probability at least } 1-\frac{M}{s_n\e_{\lag(r_n)}^2}\nonumber
\end{equation}
\begin{equation}
  \mbox{ after}	\ e_{{n}} = r_n \cdot M \cdot\left(1+\sum_{i=1}^n \frac{s_i}{r_n}\right)\label{leshift}\ \mbox{evaluations.}\nonumber
\end{equation}
\end{itemize}

Moreover, if $(s_n)$, $\lag(n)$, $(r_n)$ and $(\e_n)$ satisfy 
$\sum_{j=1}^{\infty} \frac{1}{s_j \e_{\lag(r_j)}^2} < \infty,$
 then, almost surely, there exists $n_0$ such that:
\begin{equation} \label{eq:ecartSRas}
\forall n\geq n_0,\ SR_{r_n}^{Selection} < SR_{r_n}^{Solver} + 2\e_{r_n} 
\end{equation}
\begin{equation}\label{leshifti}
	\mbox{ after }	\ e_{{n}} = r_n \cdot M \cdot\left(1+\sum_{i=1}^n \frac{s_i}{r_n}\right)\ \mbox{evaluations.}\nonumber
\end{equation}
\end{theorem}

\begin{remark}
Please notice that Equation \ref{eq:ecartSR} holds {\em{with a given probability}} whereas Equation \ref{eq:ecartSRas} holds {\em{almost surely}}. The almost sure convergence in the assumption is proved for some noisy optimization algorithms \cite{fabian}.
\end{remark}

\def\vieuxlemme{
Let $(r_n)_{n\in \N^{*}}$ and $(s_n)_{n\in \N^{*}}$ two non-decreasing sequences.
Assume that:   
\begin{itemize}
\item $\forall x \in \mathcal{D},\ Var\ f(x)\leq 1$;
\item for some positive sequence $(\e_n)_{n\in \N^{*}}$, ${\bf{P_{as}((\epsilon_n)_{n \in \N^*})}}$ (Definition \ref{hps}) is true, almost surely, there exists some $n_0\in \N^{*}$ such that :
\small
\begin{equation}\label{eq:ecartSR}
\forall n\geq n_0,\ SR_{r_n}^{Selection} < SR_{r_n}^{Solver} + 2\e_{r_n};
\end{equation}
\begin{equation}\label{eq:proba}
	\mbox{with probability at least } 1-\frac{M}{s_n\e_{\lag(r_n)}^2}
\end{equation}
\begin{equation}
	\mbox{ with }	e_{{n}}=r_n \cdot M \cdot\left(1+\sum_{i=1}^n \frac{s_i}{r_n}\right)\label{leshift}
\end{equation}evaluations.
\end{itemize}
\end{lemma}

{\bf{Interpretation:}} Within a factor $M(1+o(1))$ on the number of evaluations (see Equation \ref{leshift}; under some assumptions on $s_n$ and $r_n$) we get approximately the same simple regret (Equation \ref{eq:ecartSR}; if $\e_n$ is not too large). On a graph with x-axis the logarithm of the number of evaluations and with y-axis the logarithm of the simple-regret, there is no significant change on the y-axis and there is a $\log(M)$ shift on the x-axis.
}
\begin{proof}

First, the total number of evaluations, up to the construction of $\tilde x_{i^*_{r_n},r_n}$ at iteration $r_n$, is $e_{n}=M \left(r_n  + \sum_{i=1}^n s_i\right)$; at this point, each solver has spent $r_n$ evaluations.

{\bf{Step 1:}} Proof of Equation \ref{eq:ecartSR}.

By Chebyshev's inequality, for a given $i \in \{1,\dots,M\}$,
\begin{equation*}
\P(| \E\left[f(\tilde x_{i,\lag(r_{n})})\right] - \hat{\E}_{s_n}\left[f(\tilde x_{i,\lag(r_{n})})\right] | > \e_{\lag(r_{n})})< \frac{\Var \ f\left(\tilde x_{i,\lag(r_{n})}\right)}{s_{n}\e_{\lag(r_{n})}^{2}}\leq \frac{1}{s_n\e_{\lag(r_{n})}^{2}}.
\end{equation*}

By union bound,
\begin{equation*}
\P(\exists i \in \{1,\dots,M\} ; | \E\left[f(\tilde x_{i,\lag(r_{n})})\right] - \hat{\E}_{s_n}\left[f(\tilde x_{i,\lag(r_{n})})\right] | > \e_{\lag(r_{n})} )< \frac{M}{s_n\e_{\lag(r_{n})}^{2}}.
\end{equation*}

With notation $i^*=i^{*}_{r_n}=\underset{i \in \{1,\dots,M\}}{\arg\min}\ \hat{\E}_{s_n}\left[f(\tilde{x}_{i,\lag(r_n)})\right]$, it follows that, with probability $1-\frac{M}{s_{n}\e_{\lag(r_{n})}^{2}}$ :
\begin{eqnarray}
\E\left[f(\tilde{x}_{i^{*},\lag(r_n)})\right]&<&\hat{\E}_{s_{n}}\left[f(\tilde{x}_{i^{*},\lag(r_n)})\right]+\e_{\lag(r_{n})};\nonumber\\
\E\left[f(\tilde{x}_{i^{*},\lag(r_n)})\right]&<&\hat{\E}_{s_{n}}\left[f(\tilde{x}_{j,\lag(r_n)})\right]+\e_{\lag(r_{n})},\ \forall j \in \{1,\dots,M\} ;\nonumber\\
\E\left[f(\tilde{x}_{i^{*},\lag(r_n)})\right]&<&\E\left[f(\tilde{x}_{j,\lag(r_n)})\right]+2\e_{\lag(r_{n})},\ \forall j \in \{1,\dots,M\} ;\nonumber\\
{{\E\left[f(\tilde{x}_{i^{*},\lag(r_n)})\right]-\E\left[f(x^{*})\right]}}&<&\underset{j\in \{1,\dots,M\}}{\min}SR_{j,\lag(r_n)}+2\e_{\lag(r_{n})};\nonumber
\end{eqnarray}

So, with probability at least $1 - \frac {M}{s_n\epsilon_{\lag(r_n)}^2}$, 
\begin{equation}
	\Delta_{i^*,\lag(r_n)}<2\epsilon_{\lag(r_n)}.\label{propaz}
\end{equation}
Using ${\bf{P_{as}((\epsilon_n)_{n \in \N^*})}}$, Equation \ref{propaz} yields $\Delta_{i^*,r_n}<2\epsilon_{r_n}$ for $\lag(r_n)$ large enough, which is the expected result.

{\bf{Step 2:}} Proof of Equation \ref{eq:ecartSRas}.

We denote by $E_n$ the event ``$\Delta_{i^*,r_n}\ge 2\epsilon_{r_n}$'' (equivalent to $SR_{r_n}^{Selection} \ge SR_{r_n}^{Solver} + 2\e_{r_n}$). By Equation \ref{eq:ecartSR}, there exists $n_0 \in \N^*$ such that, $\forall n \ge n_0$, $\P(E_n) \le \frac M {s_n \e_{\lag(r_n)}^{2}}$.

Therefore 
\begin{equation*}
\sum_{j=1}^{\infty} \P(E_j) \le \sum_{j=1}^{n_0-1} \P(E_j) +M\sum_{j=n_0}^{\infty} \frac{1}{s_j \e_{\lag(r_j)}^2}<\infty.
\end{equation*}
According to Borel-Cantelli lemma, almost surely, for $n$ large enough, 
\begin{equation*}
SR_{r_n}^{Selection} < SR_{r_n}^{Solver} + 2\e_{r_n} 
\end{equation*}
and the number of evaluations is still $e_{{n}}=r_n \cdot M \cdot\left(1+\sum_{i=1}^n \frac{s_i}{r_n}\right)$.
\qed\end{proof}

We now use Proposition \ref{example} to apply Theorem \ref{rnopa} on a classical case with almost sure convergence.

\begin{appl}[$\log(M)$ shift]\label{application1}
Assume that for any solver $ i \in \{1,\dots,M \}$, the simple regret after $n \in \N^*$ evaluations is $SR_{i,n}= (1+ o(1)) \frac {C_i} {n^{\alpha_i}}$. We define $\epsilon_n= \frac C {n^{\alpha^*}}$ (where $C$ and $\alpha^{*}$ are defined as in Equations \ref{eq:defc} and \ref{eq:bestalpha}). Assume that $\forall x \in \mathcal{D}$, $\Var\ f(x) \le 1$ and that $(s_n)$, $(\lag(n))$ and $(r_n)$ satisfy:
\begin{eqnarray*}
&\sum_{j=1}^{\infty}& \frac{1}{s_j \e_{\lag(r_j)}^2} < \infty\\
\mbox{and    }&\sum_{i=1}^n& s_i=o(r_n).
\end{eqnarray*}

Then, almost surely, 
\begin{enumerate}[label=\roman*)]
\item\label{appl:logm1} for $n$ large enough, $SR_{r_n}^{Selection} < SR_{r_n}^{Solver} + 2\e_{r_n}$  after $e_{{n}}=r_n \cdot M \cdot\left(1+\sum_{i=1}^n \frac{s_i}{r_n}\right)$ function evaluations; 
\item\label{appl:logm2} for $n$ large enough,$\ SR_{r_n}^{Selection} \leq \underset{i\in SubSetOptim}\max SR_{i,r_n}$  after $e_{{n}}=r_n \cdot M \cdot\left(1+\sum_{i=1}^n \frac{s_i}{r_n}\right)$ function evaluations;
\item\label{appl:logm3} the slope of the selection regret verifies $\underset{n \rightarrow \infty} \lim \frac{\log(SR_{r_n}^{Selection})}{\log (e_n)} = -\alpha^*$.
\end{enumerate}
\end{appl}

$SR_{r_n}^{Selection}$ corresponds to the simple regret at iteration $r_n$ of the portfolio, which corresponds to $e_n=r_n \cdot M \cdot\left(1+\sum_{i=1}^n \frac{s_i}{r_n}\right)$ evaluations in the portfolio - hence the comment ``after $e_n$ function evaluations''.

\begin{proof}
By Property \ref{proposition} and Theorem \ref{rnopa}, \ref{appl:logm1} holds.

By Equation \ref{propaz}, and Property \ref{proposition}, $SR_{r_n}^{Selection}=SR_{i,r_n}$, with $i\in SubSetOptim$ and \ref{appl:logm2} follows. We obtain:
\begin{equation*}
\mbox{a.s.   }\log(SR_{r_n}^{Selection}) = \log(SR_{i,r_n} ),\text{ where $i\in SubSetOptim$.}
\end{equation*}
By the definition of $SubSetOptim$ (Equation \ref{subsetoptim}):
\begin{eqnarray*}
\underset{n \rightarrow \infty} \lim\frac{\log(SR_{r_n}^{Selection})}{\log (e_n)}&=&\underset{n \rightarrow \infty} \lim \frac{\log(SR_{i^{*},r_n})}{\log(M)+\log(r_n)+\log\left(1+\sum_{i=1}^n \frac{s_i}{r_n}\right)}\\
&=&\underset{n \rightarrow \infty} \lim \frac{\log(SR_{i^{*},r_n})}{\log(r_n)}= -\alpha^*.
\end{eqnarray*}
Hence \ref{appl:logm3} holds.
\qed\end{proof}

\def\jenlevecetrucquiestincorrect
{\color{black}
\begin{remark}\label{rem:logm}
From point \ref{appl:logm3} in Application \ref{application1}, a.s. 
\begin{eqnarray*}
\log(SR_{r_n}^{selection})&=&-\alpha^{*}(\log(M)+\log(r_n)+ o(1))\\
\mbox{and   }\log(SR_{r_n}^{solver})&=&-\alpha^{*}(\log(r_n)+o(1)).
\end{eqnarray*}
Moreover, $SR_{r_n}$ correspond to $e_n=r_n \cdot M \cdot\left(1+\sum_{i=1}^n \frac{s_i}{r_n}\right)$ evaluations in the portfolio.
Hence, we get the $\log(M)$-shift.
\end{remark}
TODO: c'est incorrect parce qu'on dit "regardez le log(M)", alors qu'il ne rime a rien,
puisqu'il est neligeable a cote de log rn, et que donc on pourrait l'enlever. les equations
qui sont ici ne montrent pas du tout le logM shift
}

\begin{ex}\label{ex:1} The following parametrization matches the conditions in Application \ref{application1}.
\begin{eqnarray*}
r_n&=&\lceil n^{3+r+r^{'}} \rceil;\\
\lag(n)&=&\lceil \log(n) \rceil;\\
s_n&=&\lceil n^{1+r^{'}} \rceil, \mbox{$r>0$ and $r^{'}\ge1$, $n \in \N^*$}.
\end{eqnarray*} 
\end{ex}

\subsubsection{The $\log(M')$-shift for INOPA}\label{logmp}

We now show that INOPA, which distributes the budget in an unfair manner, can have an improvement over NOPA. Instead of a factor $M$ (number of solvers in the portfolio), we get a factor $M'$, number of approximately optimal solvers. This is formalized in the following theorem:

 \begin{theorem}[$\log(M')$ shift]\label{th11}
Let $(r_n)_{n\in \N^{*}}$ and $(s_n)_{n\in \N^{*}}$ two non-decreasing integer sequences. 
Assume that:   
\begin{itemize}
\item $\forall x \in \mathcal{D},\ Var\ f(x)\leq 1$;
\item for some positive sequence $(\e_n)_{n\in \N^{*}}$, ${\bf{P_{as}((\epsilon_n)_{n \in \N^*})}}$ (Definition \ref{hps}) holds. 
\end{itemize}
We define $S=\{i| \exists n_0 \in \N^*, \forall n \ge n_0, \Delta_{i,n}< 2 \epsilon_{n} \}$ and $M'$ denotes the cardinality of the set $S$, i.e., $M'=|S|$. Then, there exists $n_0$ such that:
\begin{equation}
\forall n\geq n_0,\ SR_{r_n}^{Selection} < SR_{r_n}^{Solver} + 2\e_{r_n}\label{pouet}
\end{equation}
\begin{equation}\label{eq:probai}
  \mbox{with probability at least } 1-\frac{M}{s_n\e_{\lag(r_n)}^2}\nonumber
\end{equation}
\begin{equation}
  \mbox{ after}	\ e_{{n}} = r_n \cdot M' \cdot\left(1+\frac M {M'}\sum_{i=1}^n \frac{s_i}{r_n}\right)+(M-M') \lag(r_n)\ \mbox{evaluations.}\nonumber
\end{equation}

Then, if $(s_n)$, $(\lag(n))$, $(r_n)$ and $(\e_n)$ satisfy $\sum_{j=1}^{\infty} \frac{1}{s_j \e_{\lag(r_j)}^2} < \infty$, $\lag(n)=o(n)$ and $\sum_{j=1}^{n} s_j =o(r_n)$, then, almost surely, there exists $n_0$ such that:
\begin{equation} 
\forall n\geq n_0,\ SR_{r_n}^{Selection} < SR_{r_n}^{Solver} + 2\e_{r_n} \label{pouet2}
\end{equation}
\begin{equation}
	\mbox{ after }	e_{{n}} =r_n \cdot M' \cdot(1+o(1))\ \mbox{evaluations.}\nonumber
\end{equation}
\end{theorem}
\begin{proof}

For a given number of comparisons $n$, the INOPA algorithm makes the same comparisons and recommends the same value as the NOPA algorithm. Therefore all the results in Theorem \ref{rnopa} still hold, hence Eqs. \ref{pouet} and \ref{pouet2} hold - but we have to prove the number $e_n$ of evaluations.

As the algorithm chooses a solver which is not in $S$ a finite number of times, there exists $n_1$ such that, for all $n\ge n_1$, the portfolio chooses a solver in $S$ at the $n^{th}$ comparison. We consider $n_0 \ge n_1$ such that $\lag(n_0)\ge r_{n_1}$.
For $n\ge n_0$ the new number of evaluations after $n$ comparisons is:
\begin{eqnarray}
	e_n &\le& M' \cdot r_n +M\cdot\sum_{i=1}^{n}s_i + (M-M') \lag(r_n)\nonumber\\
    &=& M'\cdot r_n \cdot \left(1+\frac{M}{M^{'}}\sum_{i=1}^{n}\frac{s_i}{r_n} +\frac{M-M'}{M'}\frac{\lag(r_n)}{r_n}\right)\label{tsointsoin}\nonumber\\
    &=& M'\cdot r_n \cdot \left(1+o\left(1\right)\right).\nonumber
\end{eqnarray}
\qed\end{proof}

Using Proposition \ref{proposition}, we apply Theorem \ref{th11} above to the case of linearly convergent optimization solvers (linear in a $\log$-$\log$ scale, with $x$-axis logarithmic of the number of evaluations and $y$-axis logarithmic of the simple regret).

\begin{appl}[$\log(M') $shift]\label{application2}
{\color{black}Assume that} $\forall x \in \mathcal{D}$, $\Var\ f(x) \le 1$ and for any solver $ i \in \{1,\dots,M \}$, the simple regret after $n \in \N^*$
evaluations is $SR_{i,n}= (1+ o(1)) \frac {C_i} {n^{\alpha_i}}$. We define $\epsilon_n= \frac C {n^{\alpha^*}}$ {\color{black}with $C$ and $\alpha^*$ defined as in Eq. \ref{eq:defc} and \ref{eq:bestalpha}.} If $(s_n)_{n\in\N^*}$, $\lag(n)_{n\in\N^*}$, $(r_n)_{n\in\N^*}$ and $(\e_n)_{n\in\N^*}$ are chosen such that $\sum_{j=1}^{\infty} \frac{1}{s_j \e_{\lag(r_j)}^2} < \infty$, $\lag(n)=o(n)$ and $\sum_{j=1}^{n} s_j =o(r_n)$, then, almost surely, there exists $n_0$ such that:
\begin{enumerate}[label=\roman*)]
\item\label{logmprime1} $\forall n\ge n_0$, \ $SR_{r_n}^{Selection} < SR_{r_n}^{Solver} + 2\e_{r_n}$  after$\ e_n= M'\cdot r_n (1+o(1))$ evaluations; 
\item\label{logmprime2} $\forall n\ge n_0$, \ $\ SR_{r_n}^{Selection} \leq \underset{i\in SubSetOptim} \max SR_{i,r_n}$ after $e_n= M'\cdot r_n (1+o(1))$ evaluations;
\item\label{logmprime3} the slope of the selection regret verifies $\underset{n \rightarrow \infty} \lim \frac{\log(SR_{r_n}^{Selection})}{\log (e_n)} = -\alpha^*$.
\end{enumerate}
\end{appl}

As usual, $SR_{r_n}^{Selection}$ corresponds to the simple regret at iteration $r_n$ of the portfolio, which corresponds to $e_n=r_n \cdot M' \cdot\left(1+o(1)\right)$ evaluations in the portfolio - hence the comment ``after $e_n$ function evaluations''.

\begin{proof}
See proof of Application \ref{application1}.\qed
\end{proof}


\begin{ex}(log(M') shift)\label{ex:2}
The parametrization of Example \ref{ex:1} also matches the assumptions of Application \ref{application2}.  
\end{ex}

\subsubsection{The {\em{lag}} is necessary}\label{neclag}

In this section, we show that, if there is no {\em{lag}} (i.e., $\forall n,\ \lag(n)=n$) whenever there are only two solvers, and whenever these solvers have different slopes, the portfolio algorithm {\color{black}might} not have a satisfactory behavior, in the sense that, in the example below, it will select infinitely often the worst solver - unless $s_n$ is so large that the comparison budget is not small compared to the portfolio budget.

\begin{ex}[The {\em{lag}} is necessary]\label{ex:lagnecessary}
{\em{Let us consider the behavior of NOPA without {\em{lag}}.
We assume the following:
\begin{itemize}
\item no lag: $\forall n \in \N^*$, $\lag(r_n)=r_n$.
\item the noise is a standard normal distribution $\cal N$;
\item there are $M=2$ solvers and the two solvers of the portfolio are such that, almost surely, $SR_{i,m}= (1+ o(1)) \frac {C_i} {m^{\alpha_i}}$ after $m \in \N^*$ evaluations, $i \in \{1,2\}$, with $\alpha_1 = 1-e$ and $\alpha_2= 1-2e$, where $e\in [0,0.5)$ is a constant.
\item The comparison budget is moderate compared to the portfolio budget, in the sense that
	\begin{equation}
		s_n= O(r_n^{\beta})\label{ncku}
	\end{equation}
	with $\beta \le 2-4e$.
\end{itemize}
Then, almost surely, the portfolio will select the wrong solver infinitely often.}}
\end{ex}

\begin{proof}

Let us assume the scenario above.
Let us show that infinitely often, the portfolio will choose the wrong solver. 
Consider $Y_{1,n}$ and $Y_{2,n}$ defined by
\begin{equation*}
	Y_{i,n}= \frac1{s_n}\sum_{\ell=1}^{s_n}f(\tilde x_{i,r_n},w^{(i,\ell)})=\E_\w[f(\tilde x_{i,r_n},\w)] + Z_{i},~ i \in \{1,2\},
\end{equation*}
where 
\begin{itemize}
	\item The $w^{(i,\ell)}$ are independent Gaussian random variables,
	\item $Z_{i}=\frac1{s_n}\sum_{\ell=1}^{s_n} w^{(i,\ell)}$,
	\item $\tilde x_{i,r_n}$ is the search point recommended by solver $i$ after $r_n$ evaluations,
\end{itemize}
i.e., $Y_{i,n}$ is the average of $s_n$ evaluations of the noisy fitness function at $\tilde x_{i,r_n}$.

We denote for all $n\in \N^*$,  
\begin{equation}
\delta_{n}= \E_\w [f(\tilde{x}_{2,r_n},\w)] - \E_\w [f(\tilde{x}_{1,r_n},\w)]=SR_{2,r_n}-SR_{1,r_n}.\label{en}\nonumber
\end{equation}

$\delta_n$ is a random variable, because the expectation operator operates on $\w$; the random dependency in $(\tilde{x}_{1,r_n},\tilde{x}_{2,r_n})$ remains.

\begin{equation*}
v_{1,n}=\Var\ Y_{1,n}~\text{ and }~v_{2,n}=\Var\ Y_{2,n}.
\end{equation*}

\begin{definition}[$\mathcal{MR}_n$]\label{def:misR}
Let $\mathcal{MR}_n$ (misranking at iteration $n$) be the event ``the portfolio chooses the wrong solver at decision step $n\in \N^*$''. 
\end{definition}

\begin{remark}\label{rem:bestSR}
From the definitions of solvers $1$ and $2$, solver $1$ is the best in terms of simple regret. As a result, if $n$ is big enough, a.s.,  we get $SR_{1,r_n}< SR_{2,r_n}$, i.e., $\E_\w[f(\tilde{x}_{1,r_n},\w)] < \E_\w [f(\tilde{x}_{2,r_n},\w)]$. Then it is straightforward that, if a.s. $\E_\w[f(\tilde{x}_{2,r_n},\w)] + Z_{2} < \E_w[f(\tilde{x}_{1,r_n},\w)]+Z_{1}$, i.e., $\delta_n<Z_1-Z_2$, the portfolio chooses solver $2$ whereas solver $1$ is the best: a.s. $\mathcal{MR}_{r_n}$ occurs.   
\end{remark}

{\bf{Step 1: constructing independent events related to wrong solver choices.}}

Let us define $\delta'_n=2(C_2/r_n^{1-2e}-C_1/r_n^{1-e})$.
We have
\begin{equation}
	\delta'_n= O\left(\frac {C_2} {r_n^{1-2e}}\right) \label{minsheng}
\end{equation}

Almost surely, $\delta_n= (1+ o(1)) \frac {C_2} {r_n^{1-2e}}-(1+ o(1)) \frac {C_1} {r_n^{1-e}}$, for $n$ sufficiently large, $\delta_n<\delta'_n$. 

$\tau_n$ denotes the event: ``$Z_1-Z_2>\delta'_n$''. 
So, almost surely, for $n$ sufficiently large, the event 
$\mathcal{MR}_n$ includes the event $\tau_n$, i.e.
\begin{equation}
	\mbox{almost surely, for $n$ sufficiently large, }\tau_n\subset \mathcal{MR}_n\label{ofck}.
\end{equation}

{\bf{Step 2: Almost surely, $\tau_n$ occurs infinitely often.}}

The $\tau_n$ are independent, so we apply the converse of Borel-Cantelli lemma. First, compute the probability of $\tau_n$;
\begin{eqnarray*}
P(\tau_n)&=& P\left(\sqrt{v_{1,n}+v_{2,n}}{\cal{N}}>\delta'_{n}\right),\\
       &=&P\left({\cal N}>\frac{\delta'_n}{\sqrt{v_{1,n}+v_{2,n}}}\right)
\end{eqnarray*}
By definition of $v_{1,n}$ and $v_{2,n}$, $\exists C>0,~s.t.~ \sqrt{v_{1,n}+v_{2,n}}=\frac{C}{\sqrt{s_n}}$, so by Equation \ref{ncku}, $\exists \tilde C>0,~ \frac{1}{\sqrt{v_{1,n}+v_{2,n}}}=\frac{\sqrt{s_n}}{C}\leq \tilde C r_{n}^{\beta/2}$.\\
By \ref{minsheng}, $\exists C'>0~s.t.~\frac{\delta'_n}{\sqrt{v_{1,n}+v_{2,n}}} \leq C'\frac{C_2}{r_{n}^{1-2e}}\tilde C r_{n}^{\beta/2}$, with $\beta/2=1-2e$. Hence  $P\left({\cal N}>\frac{\delta'_n}{\sqrt{v_{1,n}+v_{2,n}}}\right) \geq P\left({\cal N}>D\right)$, with $D>0$.\\
We get $P(\tau_n)=\Omega(1)$, as the $\tau_n$ are independent, Borel-Cantelli's lemma (converse) implies that almost surely, $\tau_n$ occurs infinitely often.

{\bf{Step 3: Concluding.}}

Step 2 has shown that almost surely, $\tau_n$ occurs infinitely often.
Equation \ref{ofck} implies that this is also true for $\mathcal{MR}_n$.

Therefore, infinitely often, the wrong solver is selected.\qed
\end{proof}

\section{Experimental results}\label{xp}

This section is organized as follows.
Section \ref{rw} introduces another version of algorithms, more adapted to some particular implementations of solvers. Section \ref{sec:diffalgo} describes the different solvers contained in the portfolios and some experimental results. Section \ref{sec:samalgo} describes the similar solvers contained in the portfolios and some experimental results. 
In all tables, $CT$ refers to the computation time and NL refers to ``no {\em{lag}}''. “s” as a unit refers to “seconds”.

\subsection{Real world constraints \& introducing sharing}\label{rw}
The real world introduces constraints. Most solvers do not allow you to run one single fitness evaluation at a time, so that it becomes difficult to have exactly the same number of fitness evaluations per solver. We will here adapt Algorithm \ref{algo:nopa} for such a case; an additional change is the possible use of ``Sharing'' options (i.e., sharing information between the different solvers). The proposed algorithm is detailed in Algorithm \ref{algo:NopaRealWorld}.

\begin{algorithm}
{\color{black}
\begin{algorithmic}[1]
\Require {noisy optimization solvers $Solver_1, Solver_2,\dots, Solver_M$}
\Require {a {\em{lag}} function $\lag: \N^{*} \mapsto \N^{*}$} \Comment{Refer to Definition \ref{def:lag}}
\Require {a non-decreasing integer sequence $r_1, r_2, \dots$}{\color{black}\Comment{Periodic comparisons}}
{\color{black}\Require {a non-decreasing integer sequence $s_1, s_2, \dots$} \Comment{Number of resamplings}}
\Require {a boolean $sharing$}
\State{$n\leftarrow 1$} \Comment{Number of selections}
\State{$i^*\leftarrow null$}	 \Comment{Index of recommended solver}
\State{$x^*\leftarrow null$}	 \Comment{Recommendation}
\State{$R\leftarrow 0^M$} \Comment{Vector of number of evaluations}
\While{budget is not exhausted}
	\If{{\color{black}$\underset{i\in\{1,\dots,M\}}\min R_{i} \geq r_n$}}
		\State{$i^{*}=\underset{i \in \{1,\dots,M\}}{\arg\min} \hat \E_{s_n}[f(\tilde{x}_{i,\lag(r_{n})})]$} \Comment{Algorithm selection}
		\State{{\color{black}$x^*=\tilde x_{i^{*},R_{i^*}}$}} \Comment{Update recommendation}
		\If{$sharing$}
			\State{{\color{black}All solvers receive $x^{*}$ as next iterate}}
		\EndIf
		\State{$n\leftarrow n+1$}
	\Else
		\For{$i \in \{1,\dots,M\}$}
			\While{$R_i < r_n$}
				\State{Apply one iteration for $Solver_i$, increase $R_i$ by the number of evals. spent}
			\EndWhile
		\EndFor
	\EndIf
\EndWhile
\State{{\color{black}$x^*=\tilde x_{i^{*},R_{i^*}}$}} \Comment{Update recommendation} 
\end{algorithmic}
}
\caption{\label{algo:NopaRealWorld} Adapted version of NOPA in real world constraints.}
\end{algorithm} 

\def\oldalgo{
\begin{algorithm}
\caption{\label{algo:NopaRealWorld} Adapted version of NOPA in real world constraints.}
{\bf{Iteration 1:}} one iteration for solver $1$, one iteration for solver 2, \dots, one iteration for solver $M$.\\
{\bf{Iteration 2:}} one iteration for each solver which received less than 2 evaluations.\\

\makebox[\linewidth][c]{$\smash{\vdots}$\ \ \ $\smash{\vdots}$\ \ \ $\smash{\vdots}$ }\\

{\bf{Iteration $i$:}} one iteration for each solver which received less than $i$ evaluations.\\

\makebox[\linewidth][c]{$\smash{\vdots}$\ \ \ $\smash{\vdots}$\ \ \ $\smash{\vdots}$ }\\

{\bf{Iteration $r_1$:}} one iteration for each solver which received less than $r_1$ evaluations.\\
{\fbox{\bf{Algorithm Selection}}} Evaluate $X=\{\tilde x_{1,\lag(r_1)},\dots,\tilde x_{M,\lag(r_1)}\}$, each of them $s_1$ times; the recommendation of NOPA is $\tilde x_{i^*,m}$ for iterations $m\in\{r_1,\dots,r_2-1\}$, where $i^*=\underset{i \in \{1,\dots,M\}}\argmin \hat \E_{s_1}[f(\tilde{x}_{i,\lag(r_{1})})]$. If sharing is enabled, all solvers receive $\tilde x_{i^*,r_1}$ as next iterate.\\
{\bf{Iteration $r_1+1$:}} one iteration for each solver which received less than $r_1+1$ evaluations.\\

\makebox[\linewidth][c]{$\smash{\vdots}$\ \ \ $\smash{\vdots}$\ \ \ $\smash{\vdots}$ }\\

{\bf{Iteration $r_2$:}} one iteration for each solver which received less than $r_2$ evaluations.\\
{\fbox{\bf{Algorithm Selection}}} Evaluate $X=\{\tilde x_{1,\lag(r_2)},\dots,\tilde x_{M,\lag(r_2)}\}$, each of them $s_2$ times; the recommendation of NOPA is $\tilde x_{i^*,m}$ for iterations $m\in\{r_2,\dots,r_3-1\}$, where $i^*=\underset{i \in \{1,\dots,M\}}\argmin \hat \E_{s_2}[f(\tilde{x}_{i,\lag(r_{2})})]$. If sharing is enabled, all solvers receive $\tilde x_{i^*,r_2}$ as next iterate.\\

\makebox[\linewidth][c]{$\smash{\vdots}$\ \ \ $\smash{\vdots}$\ \ \ $\smash{\vdots}$ }\\

{\bf{Iteration $r_n$:}} one iteration for each solver which received less than $r_n$ evaluations.\\
{\fbox{\bf{Algorithm Selection}}} Evaluate $X=\{\tilde x_{1,\lag(r_n)},\dots,\tilde x_{M,\lag(r_n)}\}$, each of them $s_n)$ times; the recommendation of NOPA is $\tilde x_{i^*,m}$ for iterations $m\in\{r_n,\dots,r_{n+1}-1\}$, where $i^*=\underset{i \in \{1,\dots,M\}}\argmin \hat \E_{s_n}[f(\tilde{x}_{i,\lag(r_{n})})]$. If sharing is enabled, all solvers receive $\tilde x_{i^*,r_n}$ as next iterate.
\end{algorithm}
}

This is an adapted version of NOPA for coarse grain, i.e., in case the solvers can not be restricted to doing one fitness evaluation at a time. The adaptation of INOPA is straightforward.
We now present experiments with this adapted algorithm.

\subsection{Experiments with different solvers in the portfolio}\label{sec:diffalgo}
In this paper, unless specified otherwise, a portfolio {\em{with lag}} means that $\forall n\in\N^*, \lag(n)<n$.  For the portfolio {\em{without lag}}, at the $n^{th}$ algorithm selection, we compare $\tilde x_{i,r_n}$ instead of $\tilde x_{i,\lag(r_n)}$. This means that we choose the identity as a {\em{lag}} function $\lag$, i.e., $\forall n\in\N^*, \lag(n)=n$.

For our experiments below, we use four noisy optimization solvers and portfolios of these solvers with and without information sharing:
\begin{itemize}
	\item Solver 1: A self-adaptive ($\mu$,$\lambda$) evolution strategy with resampling as explained in Algorithm \ref{algosaes}, with parametrization $\lambda=10d$, $\mu=5d$, $K=10$, $\zeta=2$ (in dimension $d$). This solver will be termed $RSAES$ (resampling self-adaptive evolution strategy).
	\item Solver 2: Fabian's solver, as detailed in Algorithm \ref{algofabian}, with parametrization $\gamma=0.1$, $a=1$, $c=100$. This variant will be termed $Fabian1$.
	\item Solver 3: Another Fabian's solver with parametrization $\gamma=0.49$, $a=1$, $c=2$. This variant will be termed $Fabian2$.
	\item Solver 4: A version of Newton's solver adapted for black-box noisy optimization (gradients and Hessians are approximated on samplings of the objective function), as detailed in Algorithm \ref{algonewton}, with parametrization $B=1$, $\beta=2$, $A=100$, $\alpha=4$. For short this solver will be termed $Newton$.
	\item NOPA NL: NOPA of solvers 1-4 {\em{without lag}}. Functions are $r_n=\lceil n^{4.2} \rceil$, $\lag(n)=n$, $s_n=\lceil n^{2.2} \rceil$ at $n^{th}$ algorithm selection.
	\item NOPA: NOPA of solvers 1-4. Functions are $r_n=\lceil n^{4.2} \rceil$, $\lag(n)=\lceil n^{1/4.2} \rceil$, $s_n=\lceil n^{2.2} \rceil$ at $n^{th}$ algorithm selection.
	\item NOPA+S.: NOPA of solvers 1-4, with information sharing enabled. Same functions.
	\item INOPA: INOPA of solvers 1-4. Same functions.
	\item INOPA+S.: INOPA of solvers 1-4, with information sharing enabled. Same functions.
\end{itemize}
Roughly speaking, Fabian's algorithm is aimed at dealing with $z=0$ \cite{fabian}. RSAES is designed for multimodal and/or parallel optimization; it is also competitive in the unimodal setting when $z>0$ \cite{rolet2010adaptive,tcscauwet}. Newton's solver is excellent when there is very little noise \cite{tcscauwet}. 
Consistently with Equation \ref{eqssr}, we evaluate the slope of the linear convergence in $\log$-$\log$ scale by the logarithm of the average simple regret divided by the logarithm of the number of evaluations. 

\subsubsection{Experiments in unimodal case}\label{sec:unimodal}
The experiments presented in this section have been performed on 
\begin{equation}
f(x)=\|x\|^2+\|x\|^{z}{\cal{N}}\label{znoise}
\end{equation} with ${\cal{N}}$ a Gaussian standard noise.
$z = 2$ is the so-called multiplicative noise case, $z = 0$ is the additive noise case, $z = 1$ is intermediate.
The results in dimension 2 and dimension 15 are shown in Table \ref{newxptable} and \ref{newxptablebis}.

\begin{table}
\centering
\caption{\label{newxptable}{\color{black}Experiments on $f(x)=\|x\|^2+\|x\|^{z}{\cal{N}}$ in dimension $2$ with $z=0,1,2$. Numbers in this table are slopes (Eq. \ref{eqssr}). We see that the portfolio successfully keeps the best of each world (i.e. INOPA has nearly the same slope as the best of the solvers). Results are averaged over $50$ runs. ``$s$'' as a unit refers to ``seconds''. $Optimal$ means the optimum is reached for at least one run; the number of times the optimum was reached (over the 50 runs) is given between parentheses. The standard deviation is shown after $\pm$.} Table \ref{newxptablebis} presents the same results in dimension 15. ``NL'' refers to ``no lag'' cases, i.e., $\forall n\in \N^*,\lag(n)=n$.}
\scriptsize
\begin{tabular}{@{}c@{}@{}c@{}@{}c@{}@{}c@{}@{}c@{}@{}c@{}}
\hline
\multirow{2}{*} {Solver/Portfolio}  &
\multicolumn{5}{c}{Obtained slope for $d=2$, $z=0$} \\
\cline{2-6}
& $CT=10s$ & $CT=20s$ & $CT=40s$ & $CT=80s$ & $CT=160s$\\
\hline
$RSAES$ & $ -.391 \pm .009 $ & $ -.391 \pm .009 $ & $ -.396 \pm .010 $ & $ -.381 \pm .012 $ & $ -.394 \pm .012 $ \\
$Fabian1$ & $ {\bf{-1.188 \pm .012}} $ & $ {\bf{-1.188 \pm .011}} $ & $ {\bf{-1.217 \pm .010}} $ & $ {\bf{-1.241 \pm .013}} $ & $ {\bf{-1.265 \pm .011}} $ \\
$Fabian2$ & $ -.172 \pm .011 $ & $ -.161 \pm .009 $ & $ -.178 \pm .011 $ & $ -.212 \pm .015 $ & $ -.226 \pm .012 $ \\
$Newton$ & $ -.206 \pm .009 $ & $ -.206 \pm .009 $ & $ -.212 \pm .010 $ & $ -.237 \pm .011 $ & $ -.239 \pm .011 $ \\
NOPA NL & $ -.999 \pm .047 $ & $ -.870 \pm .061 $ & $ -.682 \pm .064 $ & $ -.748 \pm .066 $ & $ -.662 \pm .067 $ \\
NOPA+S. NL & $ -.210 \pm .012 $ & $ -.230 \pm .011 $ & $ -.243 \pm .013 $ & $ -.260 \pm .013 $ & $ -.255 \pm .015 $ \\
NOPA & $ -.897 \pm .054 $ & $ -.946 \pm .049 $ & $ -.835 \pm .059 $ & $ -.777 \pm .064 $ & $ -.932 \pm .058 $ \\
NOPA+S. & $ -.264 \pm .013 $ & $ -.298 \pm .015 $ & $ -.268 \pm .011 $ & $ -.304 \pm .018 $ & $ -.303 \pm .016 $ \\
INOPA & $ -.829 \pm .069 $ & $ -.950 \pm .062 $ & $ -.948 \pm .055 $ & $ -.913 \pm .058 $ & $ -.904 \pm .055 $ \\ 
INOPA+S. & $ -.703 \pm .058 $ & $ -.938 \pm .056 $ & $ -.844 \pm .055 $ & $ -.789 \pm .055 $ & $ -.776 \pm .055 $ \\ 
\hline
\hline
\multirow{2}{*} {Solver/Portfolio}  &
\multicolumn{5}{c}{Obtained slope for $d=2$, $z=1$} \\
\cline{2-6}
& $CT=10s$ & $CT=20s$ & $CT=40s$ & $CT=80s$ & $CT=160s$\\
\hline
$RSAES$ & $ -.526 \pm .013 $ & $ -.530 \pm .016 $ & $ -.507 \pm .012 $ & $ -.507 \pm .017 $ & $ -.522 \pm .014 $ \\
$Fabian1$ & $ -1.247 \pm .015 $ & $ -1.225 \pm .009 $ & $ -1.252 \pm .010 $ & $ -1.276 \pm .011 $ & $ -1.314 \pm .013 $ \\
$Fabian2$ & $ -1.785 \pm .009 $ & $ -1.755 \pm .011 $ & $ -1.782 \pm .015 $ & $ -1.777 \pm .011 $ & $ -1.738 \pm .010 $ \\
$Newton$ & $ {\bf{-2.649 \pm .010}} $ & $ {\bf{-2.605 \pm .008}} $ & $ {\bf{-2.600 \pm .011}} $ & $ {\bf{-2.547 \pm .011}} $ & $ -2.517 \pm .010 $ \\
NOPA NL & $ -1.624 \pm .011 $ & $ -1.600 \pm .011 $ & $ -1.593 \pm .016 $ & $ -1.554 \pm .013 $ & $ -1.533 \pm .014 $ \\
NOPA+S. NL & $ -1.225 \pm .013 $ & $ -1.228 \pm .013 $ & $ -1.281 \pm .014 $ & $ -1.298 \pm .015 $ & $ -1.323 \pm .017 $ \\
NOPA & $ -1.925 \pm .081 $ & $ -1.954 \pm .076 $ & $ -1.661 \pm .070 $ & $ -1.805 \pm .066 $ & $ -1.694 \pm .062 $ \\
NOPA+S. & $ -1.491 \pm .077 $ & $ -1.624 \pm .080 $ & $ -1.693 \pm .072 $ & $ -1.632 \pm .068 $ & $ -1.537 \pm .061 $ \\
INOPA & $ -2.271 \pm .062 $ & $ -2.330 \pm .061 $ & $ -2.478 \pm .033 $ & $ -2.506 \pm .047 $ & $ {\bf{-2.599 \pm .024}} $ \\ 
INOPA+S. & $ -2.013 \pm .070 $ & $ -1.927 \pm .074 $ & $ -1.987 \pm .074 $ & $ -2.120 \pm .081 $ & $ -1.856 \pm .078 $ \\ 
\hline
\hline
\multirow{2}{*} {Solver/Portfolio}  &
\multicolumn{5}{c}{Obtained slope for $d=2$, $z=2$} \\
\cline{2-6}
& $CT=10s$ & $CT=20s$ & $CT=40s$ & $CT=80s$ & $CT=160s$\\
\hline
$RSAES$ & $ -.500 \pm .013 $ & $ -.491 \pm .011 $ & $ -.484 \pm .011 $ & $ -.526 \pm .018 $ & $ -.537 \pm .015 $ \\
$Fabian1$ & $ -1.233 \pm .010 $ & $ -1.246 \pm .013 $ & $ -1.258 \pm .011 $ & $ -1.299 \pm .014 $ & $ -1.310 \pm .013 $ \\
$Fabian2$ & $ -3.173 \pm .010 $ & $ -3.175 \pm .009 $ & $ -3.141 \pm .008 $ & $ -3.120 \pm .013 $ & $ -3.073 \pm .011 $ \\
$Newton$ & $ -4.146 \pm .004 $ & $ -4.349 \pm .008 $ & $ -4.514 \pm .004 $ & $ -4.743 \pm .012 $ & $ -4.973 \pm .011 $ \\
NOPA NL & $ -2.911 \pm .009 $ & $ -2.871 \pm .010 $ & $ -2.796 \pm .011 $ & $ -2.770 \pm .012 $ & $ -2.717 \pm .014 $ \\
NOPA+S. NL & $ -2.919 \pm .011 $ & $ -2.818 \pm .050 $ & $ -2.785 \pm .047 $ & $ -2.684 \pm .056 $ & $ -2.762 \pm .016 $ \\
NOPA & $ -4.343 \pm .006 $ & $ -4.603 \pm .013 $ & $ -4.772 \pm .013 $ & $ {\bf{Optimal}} $ (1) & $ -5.103 \pm .011 $ \\
NOPA+S. & $ -4.305 \pm .041 $ & $ -4.573 \pm .011 $ & $ -4.431 \pm .091 $ & $ -4.910 \pm .048 $ & $ -5.020 \pm .059 $ \\
INOPA & $ {\bf{Optimal}} $ (1) & $ {\bf{Optimal}} $ (2) & $ -4.698 \pm .004 $ & $ -4.435 \pm .007 $ & $ -4.408 \pm 0 $ \\ 
INOPA+S. & $ {\bf{Optimal}} $ (1) & $ -3.302 \pm .116 $ & $ {\bf{Optimal}} $ (2) & $ -4.409 \pm .008 $ & $ {\bf{Optimal}} $ (35) \\ 
\hline
\end{tabular}
\end{table}

\begin{table}
\centering
\caption{\label{newxptablebis}{\color{black}Experiments on $f(x)=\|x\|^2+\|x\|^{z}{\cal{N}}$ in dimension $15$ with $z=0,1,2$. Numbers in this table are slopes (Eq. \ref{eqssr}). We see that the portfolio successfully keeps the best of each world (INOPA has nearly the same slope as the best). Results are averaged over $50$ runs. ``$s$'' as a unit refers to ``seconds''. $Optimal$ means the optimum is reached for at least one run; the number of times the optimum was reached (over the 50 runs) is given between parentheses. The standard deviation is shown after $\pm$. ``NL'' refers to ``no lag'' cases, i.e., $\forall n \in \N^*,\lag(n)=n$.}}
\scriptsize
\begin{tabular}{@{}c@{}@{}c@{}@{}c@{}@{}c@{}@{}c@{}@{}c@{}}
\hline
\multirow{2}{*} {Solver/Portfolio}  &
\multicolumn{5}{c}{Obtained slope for $d=15$, $z=0$} \\
\cline{2-6}
& $CT=10s$ & $CT=20s$ & $CT=40s$ & $CT=80s$ & $CT=160s$\\
\hline
$RSAES$ & $ .093 \pm .002 $ & $ .107 \pm .002 $ & $ .114 \pm .002 $ & $ .128 \pm .002 $ & $ .136 \pm .003 $ \\
$Fabian1$ & $ {\bf{-.825 \pm .003}} $ & $ {\bf{-.826 \pm .003}} $ & $ {\bf{-.838 \pm .003}} $ & $ {\bf{-.834 \pm .004}} $ & $ {\bf{-.835 \pm .003}} $ \\
$Fabian2$ & $ .096 \pm .003 $ & $ .108 \pm .003 $ & $ .108 \pm .003 $ & $ .114 \pm .003 $ & $ .125 \pm .003 $ \\
$Newton$ & $ -.055 \pm .002 $ & $ -.062 \pm .003 $ & $ -.070 \pm .003 $ & $ -.069 \pm .003 $ & $ -.071 \pm .003 $ \\
NOPA NL & $ -.512 \pm .046 $ & $ -.393 \pm .049 $ & $ -.377 \pm .048 $ & $ -.425 \pm .049 $ & $ -.380 \pm .046 $ \\
NOPA+S. NL & $ .026 \pm .008 $ & $ -.026 \pm .021 $ & $ -.082 \pm .025 $ & $ -.237 \pm .033 $ & $ -.410 \pm .028 $ \\
NOPA & $ -.757 \pm .003 $ & $ -.750 \pm .003 $ & $ -.747 \pm .003 $ & $ -.734 \pm .013 $ & $ -.705 \pm .018 $ \\
NOPA+S. & $ .039 \pm .007 $ & $ .019 \pm .013 $ & $ .016 \pm .019 $ & $ .005 \pm .024 $ & $ -.079 \pm .029 $ \\
INOPA & $ -.762 \pm .024 $ & $ -.768 \pm .024 $ & $ -.822 \pm .003 $ & $ -.821 \pm .003 $ & $ -.826 \pm .003 $ \\ 
INOPA+S. & $ -.484 \pm .033 $ & $ -.508 \pm .035 $ & $ -.575 \pm .038 $ & $ -.603 \pm .036 $ & $ -.499 \pm .037 $ \\ 
\hline
\hline
\multirow{2}{*} {Solver/Portfolio}  &
\multicolumn{5}{c}{Obtained slope for $d=15$, $z=1$} \\
\cline{2-6}
& $CT=10s$ & $CT=20s$ & $CT=40s$ & $CT=80s$ & $CT=160s$\\
\hline
$RSAES$ & $ .094 \pm .002 $ & $ .102 \pm .002 $ & $ .118 \pm .003 $ & $ .128 \pm .002 $ & $ .137 \pm .003 $ \\
$Fabian1$ & $ -.991 \pm .003 $ & $ -1.004 \pm .003 $ & $ -1.011 \pm .003 $ & $ -1.020 \pm .003 $ & $ -1.032 \pm .003 $ \\
$Fabian2$ & $ {\bf{-1.399 \pm .003}} $ & $ {\bf{-1.376 \pm .004}} $ & $ -1.339 \pm .003 $ & $ -1.313 \pm .003 $ & $ -1.274 \pm .004 $ \\
$Newton$ & $ -.793 \pm .099 $ & $ -.787 \pm .095 $ & $ -.959 \pm .092 $ & $ -.837 \pm .086 $ & $ -.875 \pm .078 $ \\
NOPA NL & $ -1.226 \pm .003 $ & $ -1.167 \pm .012 $ & $ -.978 \pm .013 $ & $ -.949 \pm .008 $ & $ -.943 \pm .005 $ \\
NOPA+S. NL & $ -.771 \pm .058 $ & $ -.869 \pm .065 $ & $ -.839 \pm .068 $ & $ -.860 \pm .060 $ & $ -.756 \pm .052 $ \\
NOPA & $ -.980 \pm .018 $ & $ -.962 \pm .013 $ & $ -.937 \pm .005 $ & $ -.941 \pm .005 $ & $ -.943 \pm .004 $ \\
NOPA+S. & $ -1.012 \pm .020 $ & $ -1.029 \pm .025 $ & $ -1.019 \pm .021 $ & $ -1.002 \pm .014 $ & $ -.951 \pm .010 $ \\
INOPA & $ -1.114 \pm .016 $ & $ -1.268 \pm .026 $ & $ -1.359 \pm .027 $ & $ -1.393 \pm .018 $ & $ {\bf{-1.482 \pm .026}} $ \\ 
INOPA+S. & $ -1.194 \pm .030 $ & $ -1.250 \pm .038 $ & $ {\bf{-1.556 \pm .030}} $ & $ {\bf{-1.441 \pm .041}} $ & $ -1.399 \pm .051 $ \\ 
\hline
\hline
\multirow{2}{*} {Solver/Portfolio}  &
\multicolumn{5}{c}{Obtained slope for $d=15$, $z=2$} \\
\cline{2-6}
& $CT=10s$ & $CT=20s$ & $CT=40s$ & $CT=80s$ & $CT=160s$\\
\hline
$RSAES$ & $ .094 \pm .003 $ & $ .102 \pm .002 $ & $ .113 \pm .003 $ & $ .125 \pm .003 $ & $ .146 \pm .002 $ \\
$Fabian1$ & $ -.991 \pm .003 $ & $ -1.000 \pm .003 $ & $ -1.016 \pm .003 $ & $ -1.019 \pm .003 $ & $ -1.037 \pm .004 $ \\
$Fabian2$ & $ -2.595 \pm .003 $ & $ -2.546 \pm .003 $ & $ -2.481 \pm .003 $ & $ -2.413 \pm .003 $ & $ -2.337 \pm .004 $ \\
$Newton$ & $ -2.911 \pm .279 $ & $ -2.763 \pm .291 $ & $ -2.503 \pm .285 $ & $ -2.420 \pm .265 $ & $ -2.614 \pm .240 $ \\
NOPA NL & $ -2.257 \pm .002 $ & $ -2.184 \pm .003 $ & $ -2.106 \pm .003 $ & $ -2.000 \pm .003 $ & $ -1.891 \pm .003 $ \\
NOPA+S. NL & $ -1.220 \pm .117 $ & $ -1.690 \pm .134 $ & $ -2.181 \pm .175 $ & $ -2.131 \pm .185 $ & $ -2.307 \pm .157 $ \\
NOPA & $ -2.956 \pm .121 $ & $ -2.664 \pm .107 $ & $ -2.515 \pm .095 $ & $ -2.466 \pm .090 $ & $ -2.025 \pm .050 $ \\
NOPA+S. & $ {\bf{-3.996 \pm .029}} $ & $ {\bf{-3.796 \pm .003}} $ & $ {\bf{-3.567 \pm .004}} $ & $ -3.294 \pm .003 $ & $ -2.947 \pm .026 $ \\
INOPA & $ -3.005 \pm .106 $ & $ -3.157 \pm .123 $ & $ -3.319 \pm .135 $ & $ {\bf{-3.528 \pm .144}} $ & $ {\bf{-3.751 \pm .136}} $ \\ 
INOPA+S. & $ -3.090 \pm .003 $ & $ -2.942 \pm .003 $ & $ -2.791 \pm .004 $ & $ -2.673 \pm .002 $ & $ -2.574 \pm .003 $ \\ 
\hline
\end{tabular}
\end{table}

We see on these experiments that:
\begin{itemize}
\item For $z=2$ the noise-handling version of Newton's algorithm, $Newton$, performs best among the individual solvers.
\item For $z=1$ the noise-handling version of Newton's algorithm, $Newton$, performs best in dimension $2$ and the second variant of Fabian's algorithm, $Fabian2$, performs best in  dimension $15$.
\item For $z=0$ the first variant of Fabian's algorithm, $Fabian1$, performs best (consistently with \cite{fabian}). 
\item The portfolio algorithm successfully reaches almost the same slope as the best of its solvers and sometimes outperforms all of them.
\item Portfolio with {\em{lag}} performs better than without {\em{lag}}.
\item In the case of small noise, NOPA with information sharing, termed NOPA+S., performs better than without information sharing, NOPA, in dimension $15$.
\item Results clearly show the superiority of INOPA over NOPA.
\end{itemize}
Incidentally, the poor behavior of RSAES on such a smooth case is not a surprise. Other experiments in Section \ref{sec:multimodal} show that in multimodal cases, RSAES is by far the most efficient solver among solvers above.

\subsubsection{Experiments in a multimodal setting}\label{sec:multimodal}
Experiments have been performed on a Cartpole control problem with neural network controller. The controller is a feed-forward neural network with one hidden layer of neurons. We use the same solvers as in Section \ref{sec:unimodal}.
The results are shown in Table \ref{tab:cartpole}. 
\begin{table}
\centering
\caption{\label{tab:cartpole} {\color{black}Slope of simple regret for control of the ``Cartpole'' problem using a Neural Network policy with different numbers of neurons. This test case is multimodal. These results are averaged over $50$ runs. ``$s$'' as a unit refers to ``seconds''. The standard deviation is shown after $\pm$ and shows the statistical significance of the results. Values close to $0$ correspond to cases with no convergence to the optimum, i.e., a slope zero means that the $\log$-$\log$ curve is horizontal. The test case is the one from \cite{cartpolemcts,couetouxphd}. Consistently with these references, the optimal fitness is zero.}}
\scriptsize
\begin{tabular}{@{}c@{}@{}c@{}@{}c@{}@{}c@{}@{}c@{}@{}c@{}}
\hline
\multirow{2}{*} {Solver/Portfolio}  &
\multicolumn{5}{c}{Obtained slope with $2$ neurons} \\
\cline{2-6}
& $CT=10s$ & $CT=20s$ & $CT=40s$ & $CT=80s$ & $CT=160s$\\
\hline
 $RSAES$ & $ -.503 \pm .008 $ & $ -.503 \pm .008 $ & $ -.483 \pm .007 $ & $ {\bf{-.469 \pm .006}} $ & $ {\bf{-.465 \pm .003}} $  \\
 $Fabian1$ & $ .002 \pm 0 $ & $ .002 \pm 0 $ & $ .002 \pm 0 $ & $ .002 \pm 0 $ & $ .002 \pm 0 $  \\
 $Fabian2$ & $ .002 \pm 0 $ & $ .002 \pm 0 $ & $ .002 \pm 0 $ & $ .002 \pm 0 $ & $ .002 \pm 0 $  \\
 $Newton$ & $ .002 \pm 0 $ & $ .002 \pm 0 $ & $ .002 \pm 0 $ & $ .002 \pm 0 $ & $ .002 \pm 0 $  \\
 NOPA NL & $ -.442 \pm .014 $ & $ -.469 \pm .010 $ & $ -.465 \pm .006 $ & $ -.452 \pm .005 $ & $ -.433 \pm .006 $  \\
 NOPA+S. NL & $ -.399 \pm .020 $ & $ -.465 \pm .009 $ & $ -.434 \pm .015 $ & $ -.461 \pm .007 $ & $ -.462 \pm .005 $  \\
 NOPA & $ -.480 \pm .014 $ & $ -.465 \pm .009 $ & $ -.466 \pm .008 $ & $ -.430 \pm .013 $ & $ -.431 \pm .009 $  \\
 NOPA+S. & $ -.461 \pm .017 $ & $ -.436 \pm .020 $ & $ -.475 \pm .011 $ & $ -.431 \pm .015 $ & $ -.415 \pm .015 $  \\
 INOPA & $ -.501 \pm .009 $ & $ -.468 \pm .009 $ & $ -.458 \pm .007 $ & $ -.445 \pm .006 $ & $ -.424 \pm .006 $  \\
 INOPA+S. & $ {\bf{-.524 \pm .011}} $ & $ {\bf{-.522 \pm .007}} $ & $ {\bf{-.490 \pm .006}} $ & $ {\bf{-.469 \pm .006}} $ & $ -.459 \pm .006 $  \\
\hline
\hline
\multirow{2}{*} {Solver/Portfolio}  &
\multicolumn{5}{c}{Obtained slope with $4$ neurons} \\
\cline{2-6}
& $CT=10s$ & $CT=20s$ & $CT=40s$ & $CT=80s$ & $CT=160s$\\
\hline
 $RSAES$ & $ {\bf{-.517 \pm .009}} $ & $ -.503 \pm .006 $ & $ -.481 \pm .006 $ & $ -.458 \pm .007 $ & $ -.452 \pm .004 $  \\
 $Fabian1$ & $ .002 \pm 0 $ & $ .002 \pm 0 $ & $ .002 \pm 0 $ & $ .002 \pm 0 $ & $ .002 \pm 0 $  \\
 $Fabian2$ & $ .002 \pm 0 $ & $ -.007 \pm .010 $ & $ -.016 \pm .013 $ & $ -.006 \pm .008 $ & $ -.005 \pm .007 $  \\
 $Newton$ & $ .002 \pm 0 $ & $ .002 \pm 0 $ & $ .002 \pm 0 $ & $ .002 \pm 0 $ & $ .002 \pm 0 $  \\
 NOPA NL & $ -.485 \pm .010 $ & $ -.465 \pm .011 $ & $ -.461 \pm .007 $ & $ -.460 \pm .004 $ & $ -.440 \pm .004 $  \\
 NOPA+S. NL & $ -.474 \pm .010 $ & $ -.487 \pm .008 $ & $ -.490 \pm .007 $ & $ {\bf{-.474 \pm .006}} $ & $ {\bf{-.463 \pm .004}} $  \\
 NOPA & $ -.491 \pm .009 $ & $ -.477 \pm .006 $ & $ -.459 \pm .007 $ & $ -.434 \pm .006 $ & $ -.423 \pm .006 $  \\
 NOPA+S. & $ -.505 \pm .010 $ & $ -.504 \pm .009 $ & $ {\bf{-.491 \pm .007}} $ & $ -.470 \pm .006 $ & $ -.452 \pm .006 $  \\
 INOPA & $ -.481 \pm .012 $ & $ -.480 \pm .007 $ & $ -.429 \pm .010 $ & $ -.423 \pm .008 $ & $ -.408 \pm .005 $  \\
 INOPA+S. & $ -.506 \pm .009 $ & $ {\bf{-.506 \pm .008}} $ & $ -.479 \pm .007 $ & $ -.466 \pm .006 $ & $ -.439 \pm .005 $  \\
\hline
\hline
\multirow{2}{*} {Solver/Portfolio}  &
\multicolumn{5}{c}{Obtained slope with $6$ neurons} \\
\cline{2-6}
& $CT=10s$ & $CT=20s$ & $CT=40s$ & $CT=80s$ & $CT=160s$\\
\hline
 $RSAES$ & $ -.496 \pm .008 $ & $ {\bf{-.508 \pm .008}} $ & $ -.479 \pm .007 $ & $ -.462 \pm .004 $ & $ -.439 \pm .005 $  \\
 $Fabian1$ & $ .002 \pm 0 $ & $ .002 \pm 0 $ & $ .002 \pm 0 $ & $ .002 \pm 0 $ & $ .002 \pm 0 $  \\
 $Fabian2$ & $ .002 \pm 0 $ & $ .002 \pm 0 $ & $ .002 \pm 0 $ & $ .002 \pm 0 $ & $ .002 \pm 0 $  \\
 $Newton$ & $ .002 \pm 0 $ & $ .002 \pm 0 $ & $ .002 \pm 0 $ & $ .002 \pm 0 $ & $ .002 \pm 0 $  \\
 NOPA NL & $ -.483 \pm .008 $ & $ -.477 \pm .009 $ & $ -.464 \pm .005 $ & $ -.447 \pm .005 $ & $ -.432 \pm .005 $  \\
 NOPA+S. NL & $ -.489 \pm .011 $ & $ -.496 \pm .007 $ & $ -.488 \pm .005 $ & $ {\bf{-.469 \pm .005}} $ & $ -.464 \pm .005 $  \\
 NOPA & $ -.492 \pm .021 $ & $ -.498 \pm .007 $ & $ -.477 \pm .012 $ & $ -.464 \pm .011 $ & $ -.456 \pm .004 $  \\
 NOPA+S. & $ -.430 \pm .026 $ & $ -.446 \pm .020 $ & $ -.452 \pm .014 $ & $ -.442 \pm .016 $ & $ -.417 \pm .014 $  \\
 INOPA & $ -.501 \pm .010 $ & $ -.485 \pm .010 $ & $ -.487 \pm .006 $ & $ -.463 \pm .006 $ & $ -.447 \pm .003 $  \\
 INOPA+S. & $ {\bf{-.517 \pm .009}} $ & $ -.501 \pm .010 $ & $ {\bf{-.503 \pm .006}} $ & $ -.468 \pm .005 $ & $ {\bf{-.467 \pm .003}} $  \\
\hline
\hline
\multirow{2}{*} {Solver/Portfolio}  &
\multicolumn{5}{c}{Obtained slope with $8$ neurons} \\
\cline{2-6}
& $CT=10s$ & $CT=20s$ & $CT=40s$ & $CT=80s$ & $CT=160s$\\
\hline
 $RSAES$ & $ -.493 \pm .009 $ & $ -.475 \pm .006 $ & $ -.449 \pm .006 $ & $ -.436 \pm .007 $ & $ -.420 \pm .005 $  \\
 $Fabian1$ & $ .002 \pm 0 $ & $ .002 \pm 0 $ & $ .002 \pm 0 $ & $ .002 \pm 0 $ & $ .002 \pm 0 $  \\
 $Fabian2$ & $ .002 \pm 0 $ & $ .002 \pm 0 $ & $ .002 \pm 0 $ & $ -.005 \pm .007 $ & $ .002 \pm 0 $  \\
 $Newton$ & $ .002 \pm 0 $ & $ .002 \pm 0 $ & $ .002 \pm 0 $ & $ .002 \pm 0 $ & $ .002 \pm 0 $  \\
 NOPA NL & $ -.464 \pm .010 $ & $ -.468 \pm .007 $ & $ -.431 \pm .008 $ & $ -.429 \pm .006 $ & $ -.419 \pm .006 $  \\
 NOPA+S. NL & $ -.463 \pm .008 $ & $ -.480 \pm .009 $ & $ -.485 \pm .008 $ & $ -.466 \pm .006 $ & $ -.453 \pm .005 $  \\
 NOPA & $ -.483 \pm .011 $ & $ -.485 \pm .009 $ & $ -.475 \pm .006 $ & $ -.465 \pm .005 $ & $ -.436 \pm .005 $  \\
 NOPA+S. & $ -.510 \pm .009 $ & $ {\bf{-.498 \pm .008}} $ & $ {\bf{-.508 \pm .006}} $ & $ {\bf{-.482 \pm .005}} $ & $ -.454 \pm .006 $  \\
 INOPA & $ -.488 \pm .010 $ & $ -.463 \pm .009 $ & $ -.455 \pm .007 $ & $ -.422 \pm .007 $ & $ -.426 \pm .006 $  \\
 INOPA+S. & $ {\bf{-.523 \pm .008}} $ & $ -.492 \pm .009 $ & $ -.476 \pm .007 $ & $ -.459 \pm .007 $ & $ {\bf{-.460 \pm .004}} $  \\
\hline
\end{tabular}
\end{table}             

We see on these experiments that:
\begin{itemize}
\item $RSAES$ is the most efficient individual solver.
\item The portfolio algorithm successfully reaches almost the same slope as the best of its solvers.
\item Sometimes, the portfolio outperforms the best of its solvers.
\item Results clearly show the superiority of INOPA over NOPA.
\end{itemize}

\subsection{The {\em{lag}}: experiments with different variants of Fabian's algorithm}\label{sec:samalgo}

In this section, we check if the version with {\em{lag}} disabled (i.e., $\forall n\in \N^*,\lag(n)=n$) can compete with the version with {\em{lag}} enabled (i.e., $\forall n\in \N^*,\lag(n)<n$). In previous experiments this was the case, we here focus on a case in which solvers are close to each other and check if in such a case the {\em{lag}} is beneficial.

Fabian's algorithm \cite{fabian} is a gradient descent algorithm using finite differences for approximating gradients. $\frac{a}{n}$ is the step in updates, i.e., the current estimate is updated by adding $-\frac{a}{n}\nabla f$ where $\nabla f$ is the approximate gradient. $\nabla f$ is approximated by averaging multiple redundant estimates, each of them by finite differences of size $\Theta(c/n^\gamma)$. 
Therefore, Fabian's algorithm has $3$ parameters, termed $a$, $c$ and $\gamma$. In the case of approximately quadratic functions with additive noise, Fabian's algorithm can obtain a good $s(SR)$ with small $\gamma >0$. However, $a$ and $c$ have an important non-asymptotic effect and the tuning of these $a$ and $c$ parameters is challenging. A portfolio of variants of Fabian's algorithm can help to overcome the tedious parameter tuning.

For these experiments, we consider the same noisy function as in Section \ref{sec:unimodal}. 
We use $5$ noisy optimization solvers which are variants of Fabian's algorithm, as detailed in Algorithm \ref{algofabian}, and portfolio of these solvers with and without {\em{lag}}:
\begin{itemize}
	\item Solver 1: $Fabian1$ as used in Section \ref{sec:diffalgo}.
	\item Solver 2: Fabian's solver with parametrization $\gamma=0.1$, $a=5$, $c=100$.
	\item Solver 3: Fabian's solver with parametrization $\gamma=0.1$, $a=1$, $c=200$.
	\item Solver 4: Fabian's solver with parametrization $\gamma=0.1$, $a=1$, $c=1$.
	\item Solver 5: Fabian's solver with parametrization $\gamma=0.1$, $a=1$, $c=10$.
	\item NOPA NL: Portfolio of solvers 1-5 {\em{without lag}}. Functions are $r_n=\lceil n^{4.2} \rceil$, $\lag(n)=n$, $s_n=\lceil n^{2.2} \rceil$ at $n^{th}$ algorithm selection.
	\item NOPA: NOPA of solvers 1-5. Functions are $r_n=\lceil n^{4.2} \rceil$, $\lag(n)=\lceil n^{1/4.2} \rceil$, $s_n=\lceil n^{2.2} \rceil$ at $n^{th}$ algorithm selection.
	\item NOPA+S.: NOPA of solvers 1-5, with information sharing enabled. Same functions.
	\item INOPA: INOPA of solvers 1-5. Same functions.
	\item INOPA+S.: INOPA of solvers 1-5, with information sharing enabled. Same functions.
\end{itemize}
Experiments have been performed in dimension $2$ and dimension $15$. These $5$ variants of Fabian's algorithm have asymptotically similar performance.
Table \ref{tab:xpnopalag} compares the portfolio above without {\em{lag}}, NOPA and INOPA.

{We see on these experiments that:
\begin{itemize}
\item The {\em{lag}} is usually beneficial, though this is not always the case.
\item Again, INOPA clearly outperforms NOPA.
\end{itemize}}

\begin{table}
\centering
\caption{\label{tab:xpnopalag}Experiments on $f(x)=\|x\|^2+\|x\|^{z}{\cal{N}}$ in dimension $2$ and dimension $15$ with $z=0,1,2$. Results are mean of $1000$ runs. Solvers are various parametrizations of Fabian's algorithm (see text). ``$s$'' as a unit refers to ``seconds''. The standard deviation is shown after $\pm$ and shows the statistical significance of the results. We use smaller time settings - this is because here the objective function has a negligible computation time.}
\scriptsize
\begin{tabular}{@{}c@{}@{}c@{}@{}c@{}@{}c@{}@{}c@{}@{}c@{}}
\hline
\multirow{2}{*} {Portfolio}  &
\multicolumn{5}{c}{obtained slope for $d=2$} \\
\cline{2-6}
& $  CT=0.05s $ & $ CT=0.1s $ & $ CT=0.2s $ & $ CT=0.4s $ & $ CT=0.8s $  \\
\hline
\multicolumn{6}{c}{$z=0$} \\
\hline
NOPA NL & $ -1.157 \pm .009 $ & $ -1.223 \pm .010 $ & $ -1.146 \pm .009 $ & $ -1.109 \pm .009 $ & $ -1.043 \pm .009 $ \\ 
NOPA+S. NL & $ -1.030 \pm .009 $ & $ -1.002 \pm .011 $ & $ -.807 \pm .010 $ & $ -.839 \pm .009 $ & $ -.774 \pm .007 $ \\ 
NOPA & $ -1.255 \pm .009 $ & $ -1.203 \pm .009 $ & $ -1.156 \pm .008 $ & $ -1.145 \pm .007 $ & $ -1.152 \pm .007 $ \\ 
NOPA+S. & $ -1.030 \pm .009 $ & $ -1.044 \pm .008 $ & $ -.995 \pm .009 $ & $ -.963 \pm .008 $ & $ -.926 \pm .007 $ \\ 
INOPA & $ {\bf{-1.289 \pm .010}} $ & $ -1.247 \pm .008 $ & $ {\bf{-1.201 \pm .008}} $ & $ {\bf{-1.188\pm0.008}} $ & $  {\bf{-1.153\pm0.007}} $ \\ 
INOPA+S. & $ -1.246 \pm .010 $ & $ {\bf{-1.266 \pm .008}} $ & $ -1.182 \pm .010 $ & $  -1.135\pm0.010 $ & $ -1.113\pm0.009 $ \\ 
\hline
\multicolumn{6}{c}{$z=1$} \\
\hline
NOPA NL & $ -1.529 \pm .005 $ & $ -1.471 \pm .005 $ & $ -1.455 \pm .004 $ & $ -1.414 \pm .004 $ & $ -1.381 \pm .004 $ \\ 
NOPA+S. NL & $ -1.436 \pm .006 $ & $ -1.401 \pm .006 $ & $ -1.287 \pm .006 $ & $ -1.225 \pm .005 $ & $ -1.129 \pm .005 $ \\ 
NOPA & $ -1.543 \pm .005 $ & $ -1.490 \pm .004 $ & $ -1.456 \pm .004 $ & $ -1.416 \pm .003 $ & $ -1.377 \pm .003 $ \\ 
NOPA+S. & $ -1.469 \pm .005 $ & $ -1.462 \pm .005 $ & $ -1.377 \pm .005 $ & $ -1.339 \pm .005 $ & $ -1.301 \pm .004 $ \\ 
INOPA & $ {\bf{-1.656 \pm .004}} $ & $ {\bf{-1.578 \pm .005}} $ & $ {\bf{-1.531 \pm .004}} $ & $ {\bf{-1.497 \pm .005}} $ & $ {\bf{-1.443 \pm .004}} $ \\ 
INOPA+S. & $ -1.638 \pm .005 $ & $ -1.568 \pm .005 $ & $ -1.503 \pm .005 $ & $ -1.474 \pm .005 $ & $ -1.425 \pm .005 $ \\ 
\hline
\multicolumn{6}{c}{$z=2$} \\
\hline
NOPA NL & $ -1.528 \pm .004 $ & $ -1.505 \pm .005 $ & $ -1.440 \pm .004 $ & $ -1.394 \pm .004 $ & $ -1.375 \pm .003 $ \\ 
NOPA+S. NL & $ -1.456 \pm .005 $ & $ -1.393 \pm .006 $ & $ -1.303 \pm .005 $ & $ -1.250 \pm .005 $ & $ -1.168 \pm .005 $ \\ 
NOPA & $ -1.540 \pm .005 $ & $ -1.529 \pm .004 $ & $ -1.439 \pm .004 $ & $ -1.422 \pm .004 $ & $ -1.384 \pm .003 $ \\ 
NOPA+S. & $ -1.473 \pm .006 $ & $ -1.450 \pm .005 $ & $ -1.371 \pm .004 $ & $ -1.339 \pm .004 $ & $ -1.303 \pm .004 $ \\ 
INOPA & $ {\bf{-1.681 \pm .005}} $ & $ {\bf{-1.607 \pm .004}} $ & $ {\bf{-1.530 \pm .005}} $ & $ {\bf{-1.497 \pm .004}} $ & $ {\bf{-1.439 \pm .005}} $ \\ 
INOPA+S. & $ -1.578 \pm .006 $ & $ -1.570 \pm .006 $ & $ -1.517 \pm .005 $ & $ -1.465 \pm .005 $ & $  -1.434 \pm 0.005  $ \\ 
\hline
\hline
 \multirow{2}{*} {Portfolio}  &
\multicolumn{5}{c}{obtained slope for $d=15$} \\
\cline{2-6}
& $  CT=0.05s $ & $ CT=0.1s $ & $ CT=0.2s $ & $ CT=0.4s $ & $ CT=0.8s $  \\
\hline
\multicolumn{6}{c}{$z=0$} \\
\hline
NOPA NL & $ -.673 \pm .001 $ & $ -.688 \pm .001 $ & $ -.699 \pm .001 $ & $ -.761 \pm .002 $ & $ -.779 \pm .002 $ \\ 
NOPA+S. NL & $ -.664 \pm .001 $ & $ -.684 \pm .001 $ & $ -.703 \pm .001 $ & $ -.716 \pm .001 $ & $ -.750 \pm .001 $ \\ 
NOPA & $ -.700 \pm .006 $ & $ -.609 \pm .006 $ & $ -.667 \pm .004 $ & $ -.681 \pm .005 $ & $ -.694 \pm .004 $ \\ 
NOPA+S. & $ -.591 \pm .006 $ & $ -.515 \pm .006 $ & $ -.514 \pm .005 $ & $ -.519 \pm .004 $ & $ -.527 \pm .004 $ \\ 
INOPA & $ {\bf{-.839 \pm .001}} $ & $ {\bf{-.839 \pm .001}} $ & $ {\bf{-.841 \pm .001}} $ & $ {\bf{-.840 \pm .001}} $ & $ -.839 \pm .001 $ \\ 
INOPA+S. & $ {\bf{-.839 \pm .001}} $ & $ {\bf{-.839 \pm .001}} $ & ${\bf{ -.841 \pm .001}} $ & $ -.839 \pm .001 $ & $ {\bf{-.841 \pm .001}} $ \\ 
\hline
\multicolumn{6}{c}{$z=1$} \\
\hline
NOPA NL & $ -1.004 \pm .001 $ & $ -.991 \pm .001 $ & $ -.980 \pm .001 $ & $ -.978 \pm .001 $ & $ -1.062 \pm .001 $ \\ 
NOPA+S. NL & $ -1.000 \pm .001 $ & $ -.985 \pm .001 $ & $ -.980 \pm .001 $ & $ -.990 \pm .001 $ & $ -1.066 \pm .001 $ \\ 
NOPA & $ -1.154 \pm .001 $ & $ -1.140 \pm .001 $ & $ -1.117 \pm .001 $ & $ -1.100 \pm .001 $ & $ -1.086 \pm .001 $ \\ 
NOPA+S. & $ -1.160 \pm .001 $ & $ -1.133 \pm .001 $ & $ -1.109 \pm .001 $ & $ -1.084 \pm .001 $ & $ -1.065 \pm .001 $ \\ 
INOPA & $ -1.231 \pm .001 $ & $ {\bf{-1.249 \pm .001}} $ & $ {\bf{-1.238 \pm .001}} $ & $ {\bf{-1.218 \pm .001}} $ & $ {\bf{-1.200 \pm .001}} $ \\ 
INOPA+S. & $ {\bf{-1.242 \pm .001}} $ & $ -1.198 \pm .003 $ & $ -1.169 \pm .003 $ & $ -1.151 \pm .002 $ & $ -1.131 \pm .003 $ \\ 
\hline
\multicolumn{6}{c}{$z=2$} \\
\hline
NOPA NL & $ -.999 \pm .001 $ & $ -.995 \pm .001 $ & $ -.981 \pm .001 $ & $ -.980 \pm .001 $ & $ -1.065 \pm .001 $ \\ 
NOPA+S. NL & $ -.999 \pm .001 $ & $ -.979 \pm .001 $ & $ -.973 \pm .001 $ & $ -.987 \pm .001 $ & $ -1.064 \pm .001 $ \\ 
NOPA & $ -1.174 \pm .001 $ & $ -1.135 \pm .001 $ & $ -1.119 \pm .001 $ & $ -1.101 \pm .001 $ & $ -1.083 \pm .001 $ \\ 
NOPA+S. & $ -1.152 \pm .002 $ & $ -1.130 \pm .001 $ & $ -1.103 \pm .001 $ & $ -1.080 \pm .001 $ & $ -1.061 \pm .001 $ \\ 
INOPA & $ {\bf{-1.234 \pm .001}} $ & $ {\bf{-1.251 \pm .001}} $ & $ {\bf{-1.237 \pm .001}} $ & $ {\bf{-1.219 \pm .001}} $ & $ {\bf{-1.197 \pm .001}} $ \\ 
INOPA+S. & $ -1.085 \pm .001 $ & $ -1.085 \pm .001 $ & $ -1.083 \pm .001 $ & $ -1.083 \pm .001 $ & $ -1.085 \pm .001 $ \\ 
\hline
\end{tabular}
\end{table}

{\color{black}
\subsection{Discussion of experimental results}
In short, experiments 
\begin{itemize}
	\item validate the use of portfolio (almost as good as the best solver, and sometimes better thanks to its inherent mitigation of ``bad luck runs''); we incidentally provide, with INOPA applied to several independent copies of a same solver, a principled tool for restarts for noisy optimization;
	\item validate the improvement provided by unfair budget, as shown by the improvement of INOPA vs NOPA (when no sharing is applied, i.e. in the context in which our mathematical results are proved) - more precisely, we get either very similar results (in Table \ref{tab:cartpole} and for $z=0$ or $z=2$ in Table \ref{newxptable}, INOPA and NOPA have essentially the same behavior), or a consistent improvement of INOPA vs NOPA ($z=1$ in Table \ref{newxptable} and Tables \ref{newxptablebis}, \ref{tab:xpnopalag});
	\item are less conclusive in terms of comparison ``with {\em{lag}} / without {\em{lag}}'', though on average {\em{lag}} is seemingly beneficial.
\end{itemize}
}

\section{Conclusion}\label{sec:conc}

We have seen that noisy optimization provides a very natural framework for portfolio methods. Different noisy optimization algorithms have extremely different convergence rates (different slopes) on different test cases, depending on the noise level, on the multimodalities, on the dimension (see e.g. Tables \ref{newxptable} and \ref{tab:xpnopalag}, where depending on $z$ the best solver is a variant of Fabian or Newton's algorithm; and Table \ref{tab:cartpole}, where RSAES is the best); see also \cite{liu2014meta} for {\color{black}experiments on} additional multimodal test cases. We proposed two versions of such portfolios, NOPA and INOPA, the latter using an unfair distribution of the budget. Both have theoretically the same slope as the best of their solvers, with better constants for INOPA (in particular, no shift, if $SubSetOptim$ (see Eq. \ref{ssopt}) has cardinal $1$).

We show mathematically 
an asymptotic $\log(M)$ shift when using $M$ solvers, when working on a classical $\log$-$\log$ scale (classical in noisy optimization); see Section \ref{sec:logMshift}. Contrarily to noise-free optimization (where a $\log(M)$ shift would be a trivial result), such a shift is not so easily obtained in noisy optimization. 
Importantly, it is necessary (Section \ref{neclag}), for getting the $\log(M)$ shift, that:
\begin{itemize}
\item the AS algorithm compares {\em{old}} recommendations (and selects a solver from this point of view); 
\item the portfolio recommends the {\em{current}} recommendation of this selected solver.
\end{itemize}
Additionally, we improve the bound to a $\log(M')$ shift, where $M'$ is the number of optimal solvers, using an unfair distribution of the computational budget (Section \ref{logmp}). In particular, the shift is asymptotically negligible when the optimal solver is unique. 

A careful choice of portfolio parameters (function $\lag(\cdot)$, specifying the {\em{lag}}; $r_n$, specifying the {\color{black}intervals $r_{n+1}-r_n$} between two comparisons of solvers; $s_n$, specifying the number of resamplings of recommendations for selecting the best) leads to such properties; we provide principled tools for choosing these parameters. 
Sufficient conditions are given in Theorem \ref{rnopa}, with examples thereafter.

Experiments show (i) the efficiency of portfolios for noisy optimization, as solvers have very different performances for different test cases and NOPA has performance close to the best or even better when the random initialization has a big impact; (ii) the clear and stable improvement provided by INOPA, thanks to an unfair budget distribution; (iii) that the {\em{lag}} is usually beneficial, though this is not always the case. {\color{black}Importantly, without lag, INOPA could not be defined.}

In noisy frameworks, we point out that portfolios might make sense even when optimizers are not orthogonal. Even with identical solvers, or closely related optimizers, the portfolio can mitigate the effect of unlucky random contributions. This is somehow related to restarts (i.e. multiple runs with random initializations). See Table \ref{tab:xpnopalag} for cases with very close solvers, and \cite{liu2014meta} with identical solvers.

Sharing information in portfolios of noisy optimization algorithms is not so easy. Our empirical results are mitigated; but we only tested very simple tools for sharing - just sharing the current best point. A further work consists in identifying better relevant information for sharing; maybe the estimate of the asymptotic fitness value of a solver is the most natural information for sharing; if a fitness value $A$ is already found and a solver claims that it will never do better than $A$, then we can safely stop its run and save up computational power. 

\bibliographystyle{splncs}
\bibliography{test}

\def\pasavantacceptation{
}
\newpage
\appendix
\section{Appendix: Noisy optimization algorithms}\label{newton}

We present briefly several noisy optimization algorithms. 
{\color{black}Algorithm \ref{algosaes} is a classical Self Adaptive-($\mu$,$\lambda$)-Evolution Strategy, with noise handled by resamplings. Algorithm \ref{algofabian} is a stochastic gradient method, with gradient estimated by finite differences; it is known to converge with simple regret $O(1/n)$ on smooth enough functions corrupted by additive noise \cite{fabian,shamir}. Algorithm \ref{algonewton} extends Fabian's algorithm by adding second-order information, by approximating the Hessian \cite{fabian2}.}

\begin{algorithm}
{\color{black}
\begin{algorithmic}[1]
\Require{dimension $d\in \N^*$}
\Require{population size $\lambda\in \N^*$ and number of parents $\mu \in \N^*$ with $\lambda \ge \mu$}
\Require{$K>0$}	\Comment{parameter used to compute resampling number}
\Require{$\zeta\geq 0$}	\Comment{parameter used to compute resampling number}
\Require{an initial parent $x_{1,i}\in\R^d$ and an initial $\sigma_{1,i}=1$, $i \in \{1,\dots,\mu\}$} 
\State{$n\leftarrow 1$}
\State{$\tilde x\leftarrow x_{1,1}$}\Comment{recommendation}
\While{(true)}
	\State{Generate $\lambda$ individuals $i_j$, $j\in\{1,\dots,\lambda\}$, independently using} \Comment{offspring}
	\begin{eqnarray}
		\sigma_j=\sigma_{n,mod(j-1,\mu)+1}\times\exp\left(\frac{\cal{N}}{2d}\right)\nonumber \text{ and } \ i_j=x_{n,mod(j-1,\mu)+1} + \sigma_{j}{\cal{N}}\nonumber
	\end{eqnarray}
        \State{Evaluate each of them $\lceil Kn^\zeta \rceil$ times and average their fitness values}
        \State{Define $j_1,\dots,j_\lambda$ so that }	\Comment{ranking}
		$$\hat \E_{\lceil Kn^\zeta \rceil}[f(i_{j_1})]\leq \hat \E_{\lceil Kn^\zeta \rceil}[f(i_{j_2})]\dots \leq \hat \E_{\lceil Kn^\zeta \rceil}[f(i_{j_{\lambda}})]$$
where $\hat \E_m$ denotes the average over $m$ resamplings
        \State{Compute ${x_{n+1,k}}$ and $\sigma_{n+1,k}$ using}\Comment{update}
	\begin{eqnarray*}
		\sigma_{n+1,k}={\sigma_{j_{k}}}
               \text{\ \ and\ \ } {x_{n+1,k}}=i_{j_{k}},\ k \in \{1,\dots,\mu\}
	\end{eqnarray*}
	\State{$\tilde x=i_{j_1}$} \Comment{update recommendation}
        \State{$n\leftarrow n+1$}
\EndWhile
\end{algorithmic}
\caption{\label{algosaes}Self-adaptive Evolution Strategy with resamplings. ${\cal{N}}$ denotes some independent standard Gaussian random variable and $c=mod(a,b)$ for $b>0$ means $\exists k \in \Z, (a-c)=bk \mbox{ and } 0\leq c<b$.}
}
\end{algorithm}

\begin{algorithm}
{\color{black}
\begin{algorithmic}[1]
\Require{dimension $d\in \N^*$}
\Require{$\frac12>\gamma>0$, $a>0$, $c>0$, even number of samples per axis $s$}
\Require{scales $1\ge u_1>\dots>u_{\frac{s}{2}}>0$, weights $w_1>\dots>w_{\frac{s}{2}}$ summing to 1}
\Require{an initial $x_1\in\R^d$}
\State{$n \leftarrow 1$}
\State{$\tilde x \leftarrow x_1$}\Comment{recommendation}
\While{(true)}
	\State{Compute $\sigma_n = c/n^\gamma$} \Comment{step-size}
	\State{Evaluate the gradient $g$ at $x_n$ by finite differences, averaging over $s$ samples per axis: $\forall i\in \{1,\dots,d\}, \forall j\in \{1, \dots, \frac{s}{2}\}$} \Comment{gradient estimation}
	\begin{eqnarray*}
		x_{n}^{(i,j)+}&=&x_n+u_j\sigma_ne_i\text{\ \ and\ \ }
		x_{n}^{(i,j)-}=x_n-u_j\sigma_ne_i\\
		g_{i}&=&\frac1{2\sigma_n}\sum_{j=1}^{s/2} w_j \left(f(x_{n}^{(i,j)+})-f(x_{n}^{(i,j)-})\right)
	\end{eqnarray*}
	\State{Apply $x_{n+1}=x_n-\frac{a}{n} g$} 	\Comment{next search point}
	\State{$\tilde x \leftarrow x_{n+1}$} \Comment{update recommendation}
	\State{$n\leftarrow n+1$}
\EndWhile
\end{algorithmic}
\caption{\label{algofabian}Fabian's stochastic gradient algorithm with finite differences. Fabian, in \cite{fabian}, proposes various rules for the parametrization; in the present paper, we use the following parameters. $s$ is as in Remark 5.2 in \cite{fabian}, i.e., $s$ is the minimal even number $\ge\frac{1}{2\gamma}-1$. The scales $u_i$ are $u_i=\frac{1}{i}, \forall i \in \{1,\dots,\frac{s}{2}\}$; this generalizes the choice in Example 3.3 in \cite{fabian}. The $w_i$ are computed as given in Lemma 3.1 in \cite{fabian}. $e_i$ is the $i^{th}$ vector of the standard orthonormal basis of $\R^d$.}
}
\end{algorithm}

\begin{algorithm}
{\color{black}
\begin{algorithmic}[1]
\Require{dimension $d\in \N^*$}
\Require{$A>0$, $B>0$, $\alpha>0$, $\beta>0$}
\Require{an initial $x_1\in\R^d$}
\State{$n\leftarrow 1$}
\State{$\tilde x \leftarrow x_1$}\Comment{recommendation}
\State{$\hat h\leftarrow$ identity matrix}
\While{(true)}
	\State{Compute $\sigma_n = A/n^\alpha$}\Comment{step-size}
	\For{$i=1 \mbox{ to } d$}
		\State Evaluate $g_i$ by finite differences at $x_n+\sigma_n e_i$ and $x_n-\sigma_n e_i$, averaging each evaluation over $\lceil Bn^\beta\rceil$ resamplings.
	\EndFor
		
	\For{$i=1 \mbox{ to } d$}
		\State{Evaluate $\hat h_{i,i}$ by finite differences at $x_n+\sigma_n e_i$, $x_n$ and $x_n-\sigma_n e_i$, averaging each evaluation over $\lceil Bn^\beta\rceil$ resamplings}
		\For{$j=1 \mbox{ to } d$, $j\neq i$}
			\State{Evaluate $\hat h_{i,j}$ by finite differences thanks to evaluations at each of $x_n\pm \sigma_n e_i\pm \sigma_n e_j$, averaging over $\lceil Bn^\beta/10\rceil$ resamplings} 
		\EndFor
	\EndFor
	\State{$\delta\leftarrow$ solution of $\hat h\delta=-g$} \Comment{possible next search point}
	\If{$\|\delta\|>\frac12\sigma_n$}
		\State{$\delta = \frac12\sigma_n\frac{\delta}{\|\delta\|}$}\Comment{trust region style}
	\EndIf
	\State{Apply $x_{n+1}=x_n+\delta$} 
	\State{$\tilde x \leftarrow x_{n+1}$} \Comment{update recommendation}
	\State{$n\leftarrow n+1$}
\EndWhile
\end{algorithmic}
\caption{\label{algonewton}An adaptation of Newton's algorithm for noisy objective functions, with gradient and Hessian approximated by finite differences and reevaluations. The recommendations are the $x_n$'s. $e_i$ is the $i^{th}$ vector of the standard orthonormal basis of $\R^d$.}
}
\end{algorithm}

{\color{black}
\section{Summary of notations}\label{ion}

 Notations are as follows:
 \begin{eqnarray*}
	 & & \mbox{General notations:}\\
	 \E_\w & =&\mbox{ expectation with respect to random variable $\w$.} \\
	 {\color{black}\hat {\E}_{k} X}& = &\mbox{ average over $k$ independent realizations of random variable $X$.}\\
	 & &  \\
	       & & \mbox{Notation for solvers:}\\
	 x_{n} & = & \mbox{ search point used by the solver for the $n^{th}$ evaluation.}\\
 \tilde x_{n} & = & \mbox{ recommendation given by the solver after }\\
	    & &	\mbox{the $n^{th}$ evaluation.}\\
	 SR_{n}&=&\E \left( f(\tilde x_{n})-f(x^*) \right).\mbox{ \ \ \ (simple regret)}\\
	       & & \\
	       & & \mbox{Notation for AS algorithms:}\\
	 i^*&=&\mbox{ index of the solver chosen by the AS algorithm}.\\
   {\color{black}\tilde x_{i,n}} & = & \mbox{ {\color{black}recommendation given by the solver $i$ after }}\\
   & & \mbox{{\color{black}the $n^{th}$ evaluation.}}\\
	 SR_{i,n}&=&\E \left( f(\tilde x_{i,n})-f(x^*) \right).\\
   M&=&\mbox{ number of solvers in portfolio}.\\
	 \Delta_{i,n} &=& SR_{i,n}- \underset{j \in \{1,\dots,M\}}\min SR_{j,n}.\\
	 SR_{n}^{Solvers}&=&\underset{i\in \{1,\dots M\}}{\min} SR_{i,n}.\\
	 SR_{n}^{Selection}&=&\E \left( f(\tilde x_{i^{*},n})-f(x^*) \right). 
 \end{eqnarray*}
}

\end{document}